\documentclass[11pt,reqno]{amsart}
\usepackage[margin=3.6cm]{geometry}
\usepackage{graphicx}
\usepackage{amssymb}
\usepackage{bbm}
\usepackage{color}
\usepackage{appendix}
\usepackage{svg}
\usepackage{booktabs}
\usepackage{tabularx}

\newcommand{\R}{{\mathbb R}}
\newcommand{\N}{{\mathbb N}}
\newcommand{\I}{{\mathbb I}}
\newcommand{\C}{{\mathbb C}}
\newcommand{\Z}{{\mathbb Z}}

\newcommand{\ep}{\varepsilon}

 \renewcommand{\Im} {\mathrm{Im\,}}
 \renewcommand{\Re} {\mathrm{Re\,}}

\newcommand{\WF}{\text{WF}}

\newtheorem{theorem}{Theorem}

\newtheorem{proposition}[theorem]{Proposition}
\newtheorem{lemma}[theorem]{Lemma}

\theoremstyle{definition}
\newtheorem{definition}{Definition}
\newtheorem{remark}{Remark}

\newtheorem*{Example}{Example}

\def\p{\partial}

\def\supp{\mathrm{supp}\,}

\begin{document}
\title[$L^2$ restriction bounds for analytic continuations]{$L^2$ restriction bounds for analytic continuations of quantum ergodic Laplace eigenfunctions}
\author{John A. Toth and Xiao Xiao}
\address{Department of Mathematics and Statistics, McGill University, Montr\'eal, QC, Canada}
\email{john.toth@mcgill.ca} 
\address{Department of Mathematics and Statistics, McGill University, Montr\'eal, QC, Canada}
\email{xiao.xiao5@mail.mcgill.ca } 

\begin{abstract}
We prove a quantum ergodic restriction (QER) theorem for real hypersurfaces $\Sigma \subset X,$ where $X$ is the Grauert tube associated with  a real-analytic, compact Riemannian manifold. As an application, we obtain $h$ independent upper and lower bounds for the $L^2$ - restrictions of the FBI transform of Laplace eigenfunctions restricted to $\Sigma$ satisfying certain generic geometric conditions. 
\end{abstract}

\maketitle
\nocite{*}

\section{Introduction}

Let $(M^n,g)$ be an $n$-dimensional $C^{\infty}$ compact Riemannian manifold and $\{ u_{\lambda_j} \}_{1}^{\infty}$ be a {\em quantum ergodic (QE)} sequence of $L^2$-normalized eigenfunctions where $u_{\lambda_j}$ is an eigenfunction with Laplace eigenvalue $\lambda_j^2.$  In the following, we opt for semiclassical notation and set the semiclassical parameter $h_j  = \lambda_j^{-1}.$
We recall that the sequence $\{ u_{h_j} \}_{j \in {\mathcal S}}$  with ${\mathcal S} \subset \N$  is QE if for any $a \in C^{\infty}_0(T^*M),$ 
\begin{equation} \label{QE} 
\lim_{h_j \to 0^+, \, j \in {\mathcal S}} \, \langle Op_{h_j}(a) u_{\lambda_j}, u_{\lambda_j} \rangle_{L^2} = \int_{S^*M} a d\mu_L,
\end{equation}
where $d\mu_L$ is Liouville measure on $S^*M$   \cite{Dya22}. 
 Suppose $H^{n-1} \subset M^n$ is a $C^{\infty}$ separating hypersurface with unit exterior normal $\nu.$ Then, given the normalized Cauchy data $(u_h^H,  u_{h}^{H,\nu}):= (u_h |_H, h \partial_{\nu} u_h |_H),$ there is an analogous {\em quantum ergodic restriction (QER)} theorem \cite{CTZ13}: For any QE sequence of eigenfunctions $\{u_h \},$ and any $a \in S^0(T^*H),$
 
 \begin{equation} \label{QER}
 \langle Op_h(a) u_h^{H,\nu}, u_h^{H,\nu} \rangle_{L^2(H)} + \langle ( Id + h^2 \Delta_H) Op_h(a) u_h^{H}, u_h^{H} \rangle_{L^2(H)} \sim_{h \to 0} 2 \int_{S_H^*M} a d \mu_L. \end{equation}
 
 The formula in (\ref{QER}) has many applications; These include the asymptotics of eigenfunction nodal sets \cite{JZ16}.
 
In \cite{CT24},  the authors prove a 2-microlocal version of (\ref{QER}).  To describe their result, suppose $(M,g)$ is a compact real-analytic Riemannian manifold and $M^{\mathbb{C}}_{\tau}$ is the associated Grauert tube complexification of radius $\tau\in (0,+\infty]$ (see Section \ref{tube} and \cite{Z07}). We denote the analytic continuation of $u_h$ to the tube $M_{\tau}^{\C}$ by $ u_h^{\C}$. The main result of \cite{CT24} says that, given any compact separating hypersurface $\Sigma \subset M_{\tau}^{\C}$ with exterior unit normal $\nu$ and such that $\Sigma \cap S^*M$ is an $(2n-2)$-dimensional manifold, it follows that for any $a \in C^{\infty}(\Sigma),$

\begin{equation}\label{qercd}
    \begin{split}
        \langle a(h^2\Delta_{\Sigma}+2h\nabla\rho+|\nabla\rho|^2+h\Delta\rho)e^{-\rho/h}u^{\mathbb{C}}_h,e^{-\rho/h}u^{\mathbb{C}}_h\rangle&_{L^2(\Sigma)}\\
        +\langle ah\partial_\nu (e^{-\rho/h}u^{\mathbb{C}}_h),h\partial_{\nu}(e^{-\rho/h}u^{\mathbb{C}}_h)\rangle&_{L^2(\Sigma)}\\
        \sim_{h\to 0^+}e^{1/h}&\int_{\Sigma\cap S^*M}a  \,q \,d\mu_{\Sigma}.
    \end{split}
\end{equation}
Here,    $\Delta_{\Sigma}$ is the Laplacian on $\Sigma$ induced by the K\"{a}hler metric on $M_{\tau}^{\C}$, $q \in C^{\infty}(\Sigma)$ with explicit formula given in the Appendix and $d\mu_{\Sigma}$ is the restriction of Liouville measure to $\Sigma \cap S^*M.$ 

Notational remark: we write $\nabla\rho$ for the first-order operator $D_{\nabla\rho}$ to avoid the confusion with $D=-i\partial$, often used in the semiclassical literature. The meaning is simply $D_{\nabla\rho} f=\langle \nabla\rho,\nabla f\rangle=df(\nabla\rho)$.

To describe our first main result in Theorem \ref{thm1} it is convenient to reformulate (\ref{qercd}) in terms of a specific FBI transform that is compatible with the complex structure on $M_{\tau}^{\C}$ (see Section \ref{tube} for more details).

Let $E(h):=e^{\frac{h}{2} \Delta_g}: C^{\infty}(M) \to C^{\infty}(M)$ denote the heat operator at time $h/2$ where we choose the semiclassical parameter $h^{-2} \in \text{Spec}(-\Delta_g).$ Then it is well known that \cite{GLS96} with the holomorphically continued operator $E^{\C}(h): C^{\infty}(M) \to {\mathcal O}(M_{\tau}^{\C}),$

$$
T_{hol}(h):= e^{-\rho/h} E^{\C}(h)
$$
is a semiclassical FBI transform in the sense of Sj\"{o}strand \cite{S82}, where $\rho = \frac{1}{2} |\xi|_{x}^2$ is the K\"{a}hler potential on the tube $M_{\tau}^{\C}.$ In the following,  we will abuse notation somewhat and simply write $T = T_{hol}(h).$
Since the $u_h$ are Laplace eigenfunctions, it follows that, in particular,

\begin{equation}\label{Tu_h}
    Tu_h(z)=h^{-n/4}e^{-1/2h}e^{-\rho(z)/h}u_h^{\mathbb{C}}(z),\quad  z\in M^\C_\tau.
\end{equation}

In the following, we denote the restriction of $Tu_h$ to $\Sigma$  by $T_{\Sigma} u_h := T u_h |_{\Sigma}.$

The first main result of this paper uses (\ref{qercd}) to prove the following {\em 2-microlocal quantum ergodic restriction (2MQER)} result for $T_{\Sigma} u_h:$

\begin{theorem} \label{thm1}
Let $(M,g)$ be a compact $C^{\omega}$ Riemannian manifold with Grauert tube $M_{\tau}^{\C},$    $\Sigma \subset M_{\tau}^{\C}$  a compact, separating hypersurface and $\{ u_h \}$ any sequence of $L^2$-normalized  QE Laplace eigenfunctions on $M.$ Then,  for any $a \in C^{\infty}(\Sigma)$, there exists ${\mathcal P}_{\Sigma,a}(h) \in \Psi_{sc}^{0}(\Sigma)$ such that

$$
\langle {\mathcal P}_{\Sigma, a}(h) T_{\Sigma}u_h, T_{\Sigma}u_h \rangle_{L^2(\Sigma)}  \sim_{h \to 0^+} \int_{S^*M \cap \Sigma} a\,q\,d\mu_{\Sigma}.
$$
\end{theorem}

The formula for the operator ${\mathcal P}_{\Sigma, a}(h)$ is somewhat cumbersome to write but is given explicitly in (\ref{explicit}), in which the angles $\theta,\phi$ depend on the positioning of $\Sigma$ relative to the structures $\rho,J$ of the ambient space. The precise definitions are given in the paragraph before (\ref{theta,phi}). Finally, we note that in view of (\ref{qercd}) one can write

\begin{equation} \label{basic op}
{\mathcal P}_{\Sigma,a}(h) = a \cdot {\mathcal P}_{\Sigma,1}(h).
\end{equation}

\begin{definition} In the following, we say that $\Sigma$ is in {\em general position} if the condition
\begin{equation} \label{general position}
\int_{\Sigma \cap S^*M} q \, d\mu_{\Sigma} \neq 0
\end{equation}
is satisfied.
\end{definition}

We show in the appendix (see Lemma \ref{general}) that (\ref{general position}) is satisfied for a large class of hypersurfaces in $B^*M;$ in particular, for those that are sufficiently close (in terms of K\"{a}hler distance) to the ''vertical" hypersurfaces $B_H^*M$ where $H \subset M$ is a real hypersurface of $M$.

Since the dependence of ${\mathcal P}_{a,\Sigma}$ on $a$ is trivial, it turns out that one can reformulate Theorem \ref{thm1} as an asymptotic result for multipliers.  More precisely (see Remark \ref{multiplierQER}), it follows from Theorem \ref{thm1} and Proposition \ref{multiplier} that  given a QE sequence $(u_h)$,  there exists  $f_{\Sigma} \in C^{\infty}(\Sigma)$ such that for any $a \in C^{\infty}(\Sigma)$,
\begin{equation} \label{thm1.1}
\langle a  f_{\Sigma} T_{\Sigma} u_h, T_{\Sigma} u_h \rangle_{L^2(\Sigma)}  + o(1) \| T_\Sigma u_h \|^2_{L^2(\Sigma)} = \int_{S^*M \cap \Sigma} a q d\mu_{\Sigma} + o(1),
\end{equation}
as $h \to 0+.$
We note that  the QER result in (\ref{thm1.1}) is  similar to  the formulation of  ambient QE in terms of the FBI transform $T;$  that is, for any  $a \in C^{\infty}_0(M^{\C}_\tau),$
$\langle a Tu_h, Tu_h \rangle_{L^2} = \int_{S^*M} a  |b_0|^2 d\mu_L +o(1),$  where $d\mu_L$ is Liouville measure and $ |b_0|^2= \sigma(T^*T)$ (see (\ref{key psdo})).

In Proposition \ref{WF2} we construct a compact set $\mathcal{W}_{\Sigma}$ that contains $WF_h(T_{\Sigma}u_h)$. Our subsequent results use Theorem \ref{thm1} together with the more detailed analysis of $WF_h(T_{\Sigma} u_h)\subset\mathcal{W}_{\Sigma}$ and the operators ${\mathcal P}_{\Sigma}(h):= {\mathcal P}_{\Sigma, 1}(h),$ to give asymptotic upper and lower bounds for the $L^2$-restrictions $\|T_{\Sigma} u_h \|_{L^2(\Sigma)}$. In the following, $A\Subset B$ means $\overline{A}$ compact and $\overline A\subset B$.

\begin{theorem} \label{bounds}
Let $\Sigma \subset M_{\tau}^{\C}$ be a closed separating hypersurface in general position with unit exterior normal $\nu$ and Cauchy-Riemann dual $X = J \nu \in T \Sigma$ where $J$ is the almost-complex structure  (see section \ref{tube}).  Let $\{ u_h \}$ be any QE sequence of $L^2$-normalized Laplace eigenfunctions. Then, there exist constants $ h_0>0, \, c_{\Sigma}, \,C_{\Sigma} >0$ such that for all $h \in (0,h_0],$

\begin{equation}\label{crude upper bound}
c_{\Sigma} \leq  \| T_{\Sigma} u_h \|_{L^2(\Sigma)} \leq C_{\Sigma} h^{-1/2}.
\end{equation}\

 Let  
 $ {\mathcal C}:=\{ z \in \Sigma, \, X \rho_{\Sigma}(z) =0 \}$ where $\rho_{\Sigma} = \rho |_{\Sigma}.$  Then, if
\begin{equation} \label{condition}
 U \Subset \Sigma \setminus {\mathcal C}, \end{equation} 
the upper bound is locally uniform with 
\begin{equation}\label{comparable}
\|T_{\Sigma}u_h\|_{L^2(U)}\leq C_{U}.
\end{equation}
\end{theorem} 


\begin{remark}
   We note that in Theorem \ref{bounds}, $h_0>0$ depends on the particular QE sequence $(u_h)$ along with the hypersurface $\Sigma.$  Consequently, the constants $c_{\Sigma}$ and $C_{\Sigma}$ also depend on the QE sequence as well as $\Sigma.$ The same is true for $C_{U},$ where the latter constant  depends on the open set $U \subset \Sigma.$ 
\end{remark}

\begin{remark}
We note that the critical set ${\mathcal C} \neq \emptyset$ since, in particular, $ X \rho_{\Sigma} = d\rho_{\Sigma} (X) $ vanishes at critical points of $\rho_{\Sigma}$ on the compact manifold $\Sigma.$

\end{remark}

Rewriting Theorem \ref{bounds} in terms of the complexified eigenfunctions $u_h^{\C}$ gives the following weighted $L^2$ estimates:

\begin{theorem} \label{cor}
 Let $\Sigma \subset M_{\tau}^{\C}$ be a closed separating hypersurface in general position, and $\{ u_h \}$ be any QE sequence of $L^2$ normalized eigenfunctions. Then there exist constants $ h_0>0$ and $c_{\Sigma}, C_{\Sigma} >0$ such that for all $h \in (0,h_0],$
 $$ 
 c_{\Sigma} \leq c_{n,h}\int_{\Sigma} e^{-2 \rho(z)/h} | u_h^{\C}(z)|^2 \,  dz d\overline{z}_{\Sigma} \leq C_{\Sigma} h^{-1}.
 $$
 In the region $U\Subset \Sigma\setminus \ {\mathcal C}$, the upper bound improves to a locally uniform bound
 $$
 c_{n,h}\int_{U} e^{-2 \rho(z)/h} | u_h^{\C}(z)|^2 \, dz d\overline{z}_{\Sigma} \leq C_U.
 $$
\end{theorem}
 
Here $c_{n,h}=h^{-n/2}e^{-1/h}$ is the ambient normalizing constant. We note that in Theorem \ref{cor}, $dz d\overline{z}_{\Sigma}$ denotes the restriction of the symplectic measure on $M^{\C}$ to $\Sigma.$

\subsubsection{Plan of the paper}
In section \ref{tube} we give some background on Grauert tubes and adapted FBI transforms in the sense of Sj\"{o}strand. In section \ref{WF} we give an explicit localization result for $WF_h (T_{\Sigma}u_h).$ Then, in section \ref{2mqer} we use this h-wave front localization combined with the Cauchy-Riemann equations adapted to a hypersurface $\Sigma \subset M_{\tau}^{\C}$ and the asymptotic formula in (\ref{qercd}) to prove Theorem \ref{thm1}. Finally in Section \ref{UL}, using the explicit formula for ${\mathcal P}_{\Sigma,a}(h)$ in (\ref{explicit}) together with appropriate applications of $L^2$-boundedness and G\aa rding inequality, we derive the upper and lower $L^2$-restriction bounds in Theorem \ref{bounds} and Theorem \ref{cor}. The appendix explicates the density $q$.

\subsubsection{Table of notation}

We use the identification (note the minus sign)
\[
    M_\tau^{\mathbb C}\simeq B_\tau^*M\subset T^*M,
    \qquad
    z=\kappa(x,\xi)=\exp_x(-i\xi).
\]
Thus $(x,\xi)$ are coordinates on the Grauert tube itself, whereas
$(x^*,\xi^*)$ denote the dual variables on $T^*(B_\tau^*M)$.
All dimensions below are real dimensions, and all primes mean the first $(n-1)$ coordinates, for example, $x'=(x_1,\cdots, x_{n-1})$.

\begin{center}
\renewcommand{\arraystretch}{1.2}
\begin{tabular}{@{} l c l @{}}
\toprule
Space & Dimension & Local coordinates \\
\midrule
$M$                         & $n$    & $x$ \\

$T^*M$                      & $2n$   & $(x,\xi)$ \\

$M_\tau^{\mathbb C}
 \simeq B_\tau^*M$          & $2n$   & $z$ or $(x,\xi)$ \\

$T_{\mathbb R}M_\tau^{\mathbb C}$
                             & $4n$   & $(x,\xi; x_\ast,\xi_\ast)$ \\

$S^*M$                      & $2n-1$ & $(x,\xi)$, $|\xi|_g=1$ \\


$T^*(B_\tau^*M)$            & $4n$   & $(x,\xi;x^*,\xi^*)$ \\

$T^*_\Sigma(B_\tau^*M)$     & $4n-1$ & $(\beta',0;\beta'^*,\beta_{2n}^*)$ \\

$T^*\Sigma$                 & $4n-2$ & $(\beta';\beta'^*)$ or
                                        $(\alpha_\Sigma;\alpha_\Sigma^*)$ \\

$H\subset M$                & $n-1$  & $x'$ \\

$T^*H$                      & $2n-2$ & $(x',\xi')$ \\

$B_H^*M$                    & $2n-1$ & $(x,\xi)$, $x\in H$ \\

$S_H^*M$                    & $2n-2$ & $(x,\xi)$, $x\in H$, $|\xi|_g=1$ \\
\bottomrule
\end{tabular}
\end{center}

Here $(\beta',\beta_{2n})$ are normalized Fermi coordinates near $\Sigma$,
with
\[
    \Sigma=\{\beta_{2n}=0\},\qquad \partial_{\beta_{2n}}=\nu.
\]
The canonical restriction map
\[
    \pi_\Sigma:T^*_\Sigma(B_\tau^*M)\longrightarrow T^*\Sigma
\]
is given in Fermi coordinates by
\[
    \pi_\Sigma(\beta',0;\beta'^*,\beta^*_{2n})=(\beta';\beta'^*), \quad \beta' = (\beta_1,...,\beta_{2n-1}).
\]
Finally,
\[
    J:T_{\mathbb R}M_\tau^{\mathbb C}
      \longrightarrow T_{\mathbb R}M_\tau^{\mathbb C},
    \qquad
    X:=J\nu\in T\Sigma .
\]

\subsection*{Acknowledgements}
X.X. wants to thank his supervisors John A. Toth and Dmitry Jakobson for all the help and support he received during his Ph.D. studies.

X.X. is grateful for receiving the ISM {\it Bourses pour thésards étoiles} 2025. J.T. was partially supported by the NSERC discovery grant \#RGPIN/04700-2020.

\section{Some Background} \label{tube}

\subsection{The Cauchy-Riemann equation}

The Grauert tube $M^{\mathbb{C}}_{\tau}$ of radius $\tau>0$ is the canonical complexification of a real-analytic Riemannian manifold $M$. It is an open K\"ahler manifold with many special properties (see \cite{LS91},\cite{S91},\cite{GS91},\cite{GS92} for further details). The geometry of the Grauert tube is completely determined by the underlying real Riemannian manifold. (See section 2.2.)

 On $M^{\mathbb{C}}_{\tau}$, we consider the almost-complex structure $J$. It is defined as the unique endomorphism $J\in End(T^{\R}M^{\mathbb{C}}_{\tau})$ compatible with the given symplectic form. In other words, $J$ is a smooth $2n\times 2n$ real-valued matrix field with $J^2=-I_{2n}$, that is compatible with the K\"ahler metric $\tilde{g}$ and symplectic form $\omega$. By this we mean
$$
\tilde{g}(Y,Z)=\omega(Y,JZ),\quad\forall Y,Z\in \Gamma (T_{\R}M^{\mathbb{C}}_{\tau}).
$$
As an  example, when $ M^{\C} = \mathbb{C}^n \cong T^* \R^n,$
$$
\omega=\sum_{j=1}^nd\xi_j\wedge dx_j,\quad H_f=\sum_{j=1}^n\frac{\partial f}{\partial x_j} \frac{\partial}{\partial \xi_j}-\frac{\partial f}{\partial \xi_j}\frac{\partial}{\partial x_j},
$$ 
$$
J\partial_{x_j}=\partial_{\xi_j}, \quad J\partial_{\xi_j}=-\partial_{x_j},\quad H_f=J\nabla f.
$$
The sign convention is opposite to some references such as \cite{Zw12}.

The Cauchy-Riemann equation for $u^{\mathbb{C}}_h:M^{\mathbb{C}}_{\tau}\to\mathbb{C}$ can be written as
$$
du^{\mathbb{C}}_h\circ J(Y)=idu^{\mathbb{C}}_h(Y),\quad\forall Y\in \Gamma(T_{\R}M^{\mathbb{C}}_{\tau}).
$$
Let $\Sigma \subset M^{\C}_{\tau}$ be a real, oriented, closed  hypersurface with unit outward normal vector field, $\nu.$ Then,

\begin{itemize}
    \item $J\nu$ is tangent to $\Sigma,$ since $\tilde{g}(J\nu,\nu)=\omega(J\nu,J\nu)=0$.
    \item $J\nu$ is non-vanishing, since $J$ is nondegenerate and $\nu\neq 0$.
\end{itemize}
When $u^{\mathbb{C}}_h$ is holomorphic in $M^{\C}_{\tau}$, we note that
\begin{equation}\label{CReqs}    
\begin{split}
J\nu(u^{\mathbb{C}}_h)&=i\partial_{\nu}u^{\mathbb{C}}_h\\
    -J\nu(\overline{u^{\mathbb{C}}_h})&=i\partial_{\nu}\overline{u^{\mathbb{C}}_h}.
\end{split}
\end{equation}
Here again, we have identified the vector field $J\nu$ with the first order differential operator $D_{J\nu}$ that acts on scalars.

Assume that the hypersurface $\Sigma$ has a defining function $F:M^{\C} \to \R$ with $\Sigma=\{F(z)=0\}$ and $|\nabla F|_{\tilde{g}}=1$ on $\Sigma$.
Fixing a local coordinate system in $M$, the canonical symplectic form on $T^*M$ is $\omega=dx\wedge d\xi$, and then
\begin{equation}\label{Jnu coords}
J\nu=H_{F}=\partial_xF\cdot \partial_{\xi}-\partial_{\xi}F\cdot \partial_x.
\end{equation}

In Section 4, we will write $X=J\nu$ for simplicity. It should be noted that $X$ is a real tangent vector belonging to the real tangent space of $\Sigma$ at a point $z\in \Sigma\subset M^{\mathbb{C}}_{\tau}$.

\subsection{Complexified heat kernel and the FBI transform}

\subsubsection{ Grauert tubes and analytic $h$-pseudodifferential calculus} Let $M$ be a compact, closed, real-analytic manifold of dimension $n$ and $M^{\C}$ denote a Grauert tube complex thickening of $M$ which is a totally real submanifold.  We include the example below to  help  familiarize the  reader with the notation. Subsequently, we provide the general definitions.

\begin{Example}
In the flat case $M=\mathbb{T}^n=\R^n/\Z^n$,
$$
M^\C=\C^n/\Z^n,\quad \rho=\frac{1}{2}\sum_{j=1}^n\xi_j^2,\quad \tau_{\text{max}}=+\infty,
$$
$$
\alpha=\sum_{j=1}^n\xi_jdx_j,\quad \omega=\sum_{j=1}^n dx_j\wedge d\xi_j,\quad \tilde{g}=\sum_{j=1}^ndx_j^2+d\xi_j^2
$$
$$
J=\begin{pmatrix}
    0&-I\\I&0
\end{pmatrix},\quad  \Xi=\sum_j\xi_j\p_{\xi_j}, \quad\Delta_{\bar\partial}=2 \sum_{j=1}^n \big(\p_{x_j}^2+\p_{\xi_j}^2 \big).
$$
\end{Example}

By Bruhat-Whitney, there exists a {\em maximal} Grauert tube radius $\tau_{\max} >0$ \cite{Z07} such that for any $ \tau \leq \tau_{\max}$, the complex manifold $M^{\C}$ can be identified with $B^*_{\tau} := \{ (x, \xi) \in T^*M; \sqrt{\rho}(x,\xi) \leq \tau \}$ where $\sqrt{2\rho} = |\xi|_g$ is the exhaustion function using the complex geodesic exponential map $ \kappa : B^*_{\tau} \rightarrow  M^{\C}  $ with  $\kappa(x,\xi) = \exp_{x}(- i \xi).$ From now on, we fix $\tau \in (0, \tau_{\max}).$
Under this identification, we  let $z$ denote local complex coordinates in $B^*_{\tau}$ and recall that $B_{\tau}^*$ is also naturally a K\"ahler manifold with potential function $\rho$ with associated symplectic form
$ -i\partial \overline{\partial} \rho = \omega.$  The complex K\"ahler, symplectic and Riemannian structures are all linked via the isomorphism $\kappa: B_{\tau}^*M \to M^{\mathbb C}.$ Denote the almost complex structure by $J: T^{\R} M^{\mathbb{C}} \to T^{\R} M^{\mathbb{C}}$.
\begin{eqnarray} \label{grauertformula}
\omega = -i\partial \overline{\partial} \rho,\quad \omega = d \alpha, \quad \alpha = \Im \overline{\partial} \rho,
\end{eqnarray}
where the strictly plurisubharmonic function $\rho$ solves the homogeneous Monge-Ampere equation \cite{GS91} 
  
$$ \big(  \overline{\partial} \partial \sqrt{\rho} \big)^n (z) = 0, \quad z \in M^{\mathbb C}\setminus M.$$
  
The $\kappa$-corresponding objects  on $B_{\tau}^*M$ are given by
\begin{eqnarray} \label{grauertformula2}
  \kappa^*\omega  = \sum_{j=1}^n dx_j \wedge d\xi_j = \frac{1}{i} \kappa^* \partial \overline{\partial} \rho, \,\,\, \, \kappa^* \omega = d \alpha, \,\, \alpha = \sum_i \xi_i dx_i,\nonumber \\
 \kappa^*\rho(x,\xi) = \frac{1}{2} |\xi|_x^2 = \frac{1}{2} g^{ij}(x) \xi_i \xi_j. \hspace{1in}
 \end{eqnarray}
 
 In the following, we will freely identify $B_{\tau}^*M$ and $M^{\mathbb C}$ and drop reference to the isomorphism  $\kappa$ when the context is clear. As was pointed out in the previous section, the Riemannian K\"ahler metric $\tilde{g}$ on $B_\tau^* M$ associated with $\omega$ is given by
  $$ \tilde{g}(u,v) = \omega (u, Jv)$$
  where $J$ is the almost complex structure on $B_\tau^*M$ induced by $\kappa.$

 Fix $p_0 \in M$ and let $x: U \to {\mathbb R}^n$ be geodesic normal coordinates centered at $p_0$ with $x(p_0) = 0.$ Then since $J_{(p_0,\xi)} (\partial_{x_j}) = \partial_{\xi_j}$ and $J_{(p_0,\xi)}(\partial_{\xi_j}) = - \partial_{x_j}$ and the base metric $g^{ij}(x) = \delta_j^i + O(|x|^2)$ (in particular, $\partial_{x_j} g^{kl}(0) = 0$),  it follows that 
 $ \tilde{g}_{(p_0,\xi)} = |dx|_{(p_0,\xi)}^2 + |d\xi|_{(p_0,\xi)}^2.$   As a result,
 $$ \nabla_{\tilde{g}} \rho (p_0,\xi) =   \sum_{j=1}^{n} \partial_{x_j}  \big( \frac{1}{2} g^{kl}(x) \xi_k \xi_l \big) |_{x=0}  \partial_{x_j} + \partial_{\xi_j} \big( \frac{1}{2} g^{kl}(x) \xi_k \xi_l \big)|_{x=0} \partial_{\xi_j}   = \sum_{j} \xi_j \partial_{\xi_j}.$$
 
 Given that  $\omega$ in (\ref{grauertformula}) is non-degenerate with $\omega = d\alpha$ there is a unique invariant  vector field $\Xi$ solving  $\iota_\Xi \omega = \alpha.$ Moreover (see \cite{GS91} section 5), $\Xi$ satisfies 
 $\Xi \rho = 2 \rho$ and $\kappa^* \Xi = \xi \cdot \partial_{\xi}$. Since $\xi \cdot \partial_{\xi}$ and $\nabla_{\tilde{g}} \kappa^*\rho$ are consequently both invariant vector fields on $B_\tau^*M$ which agree at $(p_0,\xi)$ in geodesic normal coordinates, they must agree in all local coordinates $x$ near $p_0$. Since $p_0 \in M$ is arbitrary, by making the usual identification of $B_\tau^*M$ with $M^{\C}$, it follows that
 
    \begin{equation} \label{grad}
\nabla_{\tilde{g}} \rho (x,\xi)  = \sum_{j=1}^n \xi_j \partial_{\xi_j}, \quad (x,\xi) \in B_{\tau}^*M.
\end{equation}

From (\ref{grad}) and the argument above, it also follows that
 \begin{equation} \label{gradnorm}
\| \nabla_{\tilde{g}} \rho \|_{\tilde{g}}^2 = 2 \rho.
\end{equation}

The associated K\"ahler Laplacian is 
$$ \Delta_{\overline{\partial}} = -\overline{\partial}^* \overline{\partial} = 2 \Delta_{\tilde{g}}$$
where the latter denotes the Riemannian Laplacian with respect to $\tilde{g}$ on $B_\tau^*M.$
In the following, to simplify notation, we will write $\nabla := \nabla_{\tilde{g}}$ and $\Delta := \Delta_{\overline{\partial}}.$

Let $ -h^2 \Delta_{\overline{\partial} }: C^{\infty}_0(B^*_{\tau}) \to C^{\infty}_0(B^*_{\tau})$ denote the semiclassical K\"ahler Laplacian with $-h^2 \Delta_{\overline{\partial}} = - 2 h^2 \Delta_{\tilde{g}}.$  By possibly rescaling the semiclassical parameter $h$ we assume without loss of generality that the characteristic manifold $\{p=0\}\subset B^*_\tau$. In our situation, the assumption is $\{ |\xi|^2_g=1\} \subset B^*_{\tau}$.

\subsubsection{Fermi coordinates near a hypersurface $\Sigma \subset B_{\tau}^*M$} \label{FERMI}

Given a smooth oriented hypersurface $\Sigma \subset B_{\tau}^*M$ we let $(\beta',\beta_{2n}): U_{\Sigma} \to {\mathbb R}^{2n}$ be
normalized Fermi coordinates in a tubular neighbourhood $U_{\Sigma}$ of $\Sigma$ with $\Sigma = \{\beta_{2n}= 0 \}$ and
$\partial_{\beta_{2n}}$ the unit exterior normal to $\Sigma.$ In terms of these coordinates the conjugated Laplacian $ |\tilde{g}|^{1/4} \Delta_{\tilde{g}} |\tilde{g}|^{-1/4}$ can be written in the form
\begin{equation} \label{fermi expansion}
|\tilde{g}|^{1/4} (-h^2 \Delta_{\tilde{g}} ) \, |\tilde{g}|^{-1/4}= (h D_{\beta_{2n}})^2 + R(\beta_{2n},\beta'; hD_{\beta'}),
\end{equation}

where $R(\beta_{2n},\beta',hD_{\beta'})$ is a second-order $h$-differential operator in the tangential $\beta'$-variables and 
$R(0,\beta', hD_{\beta'}) =  -h^2 \Delta_{\Sigma}$ where $\Delta_{\Sigma}$ is the Riemannian Laplacian on the hypersurface $\Sigma$  induced by the metric $\tilde{g}.$ In the following, we abuse notation and denote the conjugated Laplacian simply by $-h^2 \Delta_{\tilde{g}}$ and $| \tilde{g}|^{1/4} u_h$ by $u_h.$

\subsubsection{FBI transform} \label{SS:FBI}
Let $U\subset T^*M$ be open. Following \cite{S82}, we say that $a \in S^{m,k}_{cla}(U)$  provided $a \sim h^{-m} (a_0 + h a_1 + \dots)$ in the sense that
\begin{equation*}
\partial_{x}^k \partial_{\xi}^l \overline{\partial}_{(x,\xi)} a = O_{k,l}(1) e^{- \langle \xi \rangle/Ch}, \quad (x,\xi)\in U, 
\end{equation*}
and for $(x , \xi) \in U$,
\begin{equation*}
  \Big| a - h^{-m} \sum_{0 \leq j \leq \langle \xi \rangle/C_0 h} h^{j} a_j \Big| = O(1) e^{- \langle \xi \rangle/C_1 h},\quad
 |a_j| \leq C_0 C^{j} \, j ! \, \langle \xi \rangle^{k-j}.
 \end{equation*}
We sometimes write $S^{m,k}_{cla}=S^{m,k}_{cla}(T^*M)$. The symbol $a \in S^{m,k}_{cla}$ is $h$-elliptic provided
$ |a(x,\xi)| \geq C h^{-m} \langle  \xi \rangle^k$ for all $(x,\xi) \in T^*M.$ In the smooth non-analytic case, we say that $a \in S^{m,k}_{cl}(T^*M)$ if $a \sim h^{-m} (a_0 +h a_1 + \cdots)$ in the (standard) sense that $ a - h^{-m} \sum_{j=0}^M a_j h^j \in S^{k -j} $ where 
$S^{k}:= \{ q \in C^{\infty}(T^*M);   |\partial_{x}^{\alpha} \partial_{\xi}^{\beta} q | = O(\langle \xi \rangle^{k-|\beta|}) \}.$

As in \cite{S82}, given an $h$-elliptic, semiclassical analytic symbol $a \in S^{3n/4,n/4}_{cla}(M \times (0,h_0]),$  we consider an intrinsic FBI transform $T(h):C^{\infty}(M) \to C^{\infty}(T^*M)$ of the form

\begin{equation} \label{FBI}
T u(x,\xi;h) = \int_{M} e^{i\phi(x,\xi,y)/h}  a(x,\xi,y,h) \tilde{\chi}( x, y) u(y) \, dy \end{equation}
In (\ref{FBI}), the cutoff $\tilde{\chi} \in C^{\infty}_{0}(M \times M)$ is supported in a  small fixed neighbourhood of $\text{diag}(M) = \{ (x,x) \in M \times M \}.$ The phase function is required to satisfy $\phi(x,\xi,x) = 0, \, \partial_y \phi(x,\xi,x) = - \xi$ and
$$ \Im (\partial_y^2 \phi)(x,\xi, x) \sim | \langle \xi \rangle | \, Id.$$

In particular, it follows that the phase $\phi$ satisfies
\begin{eqnarray} \label{phase}
\Re \phi(x,\xi,y) = \langle x-y, \xi \rangle + O(|x-y|^2 \langle \xi \rangle), \nonumber \\
\Im \phi(x,\xi,y)  = \frac{1}{2} |x-y|^2 \big( 1 + O(|x-y|) \big)  \langle \xi \rangle.
\end{eqnarray}

Given $T(h) :C^{\infty}(M) \to C^{\infty}(T^*M)$, it follows by an analytic stationary-phase argument \cite{S82} that one can construct an operator $S(h): C^{\infty}(T^*M) \to C^{\infty}(M)$ of the form
\begin{equation} \label{left}
 S v(x;h) = \int_{T^*M} e^{-i  \, \overline{\phi(y,\xi,x)}  /h} d(x,y,\xi,h) v(y,\xi) \, dy d\xi \end{equation}
with $d \in S^{3n/4,n/4}_{cla}$ such that $S(h)$  is a left-parametrix for $T(h)$ in the sense that

\begin{equation} \label{leftparametrix}
S(h) T(h) = Id + R(h),\qquad\partial_{x}^{\alpha} \partial_{y}^{\beta} R(x,y,h) = O_{\alpha, \beta}(e^{-C/h}).
\end{equation}

We also note that with the normalizations in (\ref{FBI}) an application of analytic stationary phase as in (\ref{left}) shows that there exists $e_0 \in S^{0}_{cla}$  h-elliptic such that 
$$ T^*(h) e_0 T(h) = S(h)  T(h) +O(h)_{L^2 \to L^2},$$
and consequently, it follows that
$$ \| T u_h \|_{L^2}  \approx \| u_h \|_{L^2} \approx 1.$$

We say that an operator $P(h)$ is an \emph{analytic $h$-pseudodifferential operator (analytic $h$psdo) of order $m,k$} on $M$ (i.e. $ P \in \Psi^{m,k}_{cla}(M)$) if for
 $p \in S^{m,k}_{cla}$ with $p \sim \sum_{j=0}^\infty p_j h^{j-m+k},$  $p_j \in S^{m-j}_{cla},$
 \begin{equation} \label{psdo}
 P(h) = S(h) p T(h) +O(e^{-C/h})_{L^2 \to L^2}. \end{equation}
 
  The kernel of $P(h)$ can then be written as $P(x,y;h)=K(x,y;h)+R(x,y;h)$ where for all $\alpha,\beta$,
$$|\partial_x^\alpha \partial_y^\beta R(x,y)|\leq C_{\alpha\beta} e^{-c_{\alpha\beta}/h}, \,\,\, c_{\alpha \beta} >0,$$
and
\begin{equation}
K(x,y;h)=\frac{1}{(2\pi h)^n}\int e^{\frac{i}{h}\langle x-y,\xi \rangle} e^{-|x-y|^2 \langle \xi \rangle /h} \,  \tilde{p}(x,\xi,h) \,d\xi
\end{equation}
where  $\tilde p \in S^{m,k}_{cla}$ with $\tilde{p}_0 = p_0.$  We use the standard notation $P(h)= p(x,hD) \in \Psi^{m,k}_{cla}(M)$ for the $h$-quantization in (\ref{psdo}). 
In the smooth case, where $p \in S^{m,k}_{cl},$ we will also define $P(h)$ as in (\ref{psdo}) and write $P \in \Psi^{m,k}_{cl}(M).$ Moreover, in the special case where $k=0,$ we simply write $\Psi^{m}_{h}:= \Psi^{m,0}_{cl}$ in the following.





It is convenient to choose here a particular FBI transform, $T_{hol}(h): C^{\infty}(M) \to C^{\infty}(B_{\tau}^*M)$ that is compatible with the complex structure in the Grauert tube $B_{\tau}^*M.$ This transform is readily described  in terms of the holomorphic continuation of the heat operator $e^{t \Delta_g}$ at time $t = h/2.$ 

We briefly recall here some background on the operator $T_{hol}(h): C^{\infty}(M) \to C^{\infty}(M_{\tau}^{\C})$ and refer the reader to \cite{GLS96} and \cite{GT19} for further details.

\subsubsection{Complexified heat operator on closed, compact manifolds} \label{heat}
Consider the heat operator of $(M,g)$ defined at time $h/2$ by $$E_{h}=e^{\frac{h}{2}\Delta_g}:C^{\infty}(M) \to C^{\infty}(M).$$

By a result of Zelditch \cite[Section 11.1]{Z12}, the maximal geometric tube radius $\tau_{\max}$ agrees with the maximal analytic tube radius in the sense that for all $ 0<\tau < \tau_{\max}$, all the eigenfunctions $\varphi_j$ extend holomorphically to $M_\tau^\C$ (see also \cite[Prop. 2.1]{GLS96}). In particular, the kernel $E(\cdot,\cdot;h)$ admits a holomorphic extension to $B^*_{\tau} M \times B^*_{\tau}M$
for all $0<\tau < \tau_{\max}$  and $h \in (0,1)$, \cite[Prop. 2.4]{GLS96}. We denote the complexification by $E_h^\C( \cdot, \cdot)$.
To recall asymptotics for $E^{\C}_h$ we note that the squared geodesic distance on $M$
$$r^2(\cdot, \cdot): M \times M \to {\mathbb R}$$
holomorphically continues in both variables to $M_{\tau}^{\C} \times M_{\tau}^{\C}$ in a straightforward fashion. 
More precisely, for every $0<\tau<\tau_{\max}$, there exists a connected open neighbourhood $\tilde \Delta \subset M_\tau^\C \times M_\tau^\C$ of the diagonal $\Delta \subset M\times M$ to which $r^2(\cdot, \cdot)$ can be holomorphically extended \cite[Corollary 1.24]{GLS96}. We denote the extension
by $r_\C^2(\cdot,\cdot) \in \mathcal O(\tilde \Delta)$. Moreover, one can easily recover the exhaustion function $\sqrt{\rho_g}(z)$ from $r_{\C}$; indeed, 
$\rho_g(z)=-r^2_\C(z, \bar{z})$
for all $z \in B^*_{\tau}M$.

The basic asymptotic behaviour of $E_h^{\C}(z,y)$ with $(z,y) \in B^*_{\tau}M \times M$ is studied in \cite{GLS96}. In particular,

\begin{equation}\label{hol heat}
E_h^\C(z,y)=e^{-\frac{r^2_\C(z,y)}{2h}} b^\C(z,y; h) + O (e^{-\beta/h}), \quad (z,y) \in B_{\tau}^*M \times M.
\end{equation}
Here, $ \beta>0$ is a constant depending on $(M,g,\tau)$ and
\begin{equation} \label{heat symbol}
b^\C \sim \sum_{k=0}^{\infty} b_{k}^{\C}  h^{k - \frac{n}{2} } \in S^{n/2,0}_{cla}; \,\,\, b_k^{\C} \in S^{0,0}_{cla}, \,\, k=0,1,2,...,
\end{equation}
  where the $b_k^\C$'s denote the analytic continuation of the coefficients appearing in the formal solution
of the heat equation on $(M,g).$ In the following, to simplify notation,  we will simply write $b_k = b_k^{\C}; \, k=0,1,2,...$  for the symbols in the expansion (\ref{heat symbol}).
 \medskip

The K\"ahler potential
\begin{equation} \label{Kahler potential}
2\rho(z) =  \Re r^2_{\C}(z, \Re z) = \frac{1}{4} r_{\C}^2(z, \bar{z}) =  |\xi|_{x}^2 \end{equation}
where, $z = \exp_{x} ( -i \xi).$

Using (\ref{Kahler potential}) and the expansion in (\ref{hol heat}) it is proved in \cite [Theorem 0.1]{GLS96}  that the operator $T_{hol}(h): C^{\infty}(
M) \to C^{\infty}(M_{\tau}^{\C})$ given by
\begin{equation} \label{holFBI}
T_{hol} \phi_h (z) = h^{-n/4} \int_{M} e^{ [ - r_{\C}^2(z,y)/2 - \rho(z)]/h} b^{\C}(z,y,h) \chi(x,y) \phi_h(y) dy, \quad z \in B_{\tau}^* \end{equation}
is also an FBI transform in the sense of \eqref{FBI} with $h$-elliptic amplitude  $b \in S^{n/2,0}_{cla}$ and phase function
$\phi(z,y) = i  \Big( \, \frac{ r_{\C}^2(z,y)}{2} + \rho(z) \, \Big). $
In \eqref{holFBI} the multiplicative factor $h^{-n/4}$ is added to ensure $L^2$-normalization so that $\| T_{hol} \phi_h \|_{L^2(M_{\tau}^{\C})} \approx  1.$

Since $u_h$ are eigenfunctions of the Riemannian Laplacian on $(M,g)$ with eigenvalue $h^{-2}$ it follows by analytic continuation that 

$$ e^{-1/2h} u^{\C}(z) = E^{\C}(h) u_h (z); \quad z \in B_{\tau}^*M.$$

Consequently, in view of (\ref{holFBI}),
\begin{equation} \label{comp}
T_{hol}(h) u_h (z) = h^{-n/4}e^{-\rho(z)/h} E^{\C}(h) u_h (z) = h^{-n/4}e^{-1/2h} e^{-\rho(z)/h} u_h^{\C}(z). \end{equation} \

Using (\ref{hol heat}), it follows that the left parametrix $S_{hol}(h)$ in (\ref{left}) satisfies
$$  T^*_{hol}(h) |b_0|^{-2} T_{hol}(h) = S_{hol}(h) T_{hol}(h)  + O(h)_{L^2 \to L^2},$$
where $b_0 \in S^{0}_{cla}$ is $h$-elliptic principal symbol in (\ref{hol heat}).
Since $S(h) a T(h) = Op_{h}(a) + O(h^{\infty}),$ the semiclassical anti-Wick quantization, it follows by an application of standard 
$h$-pseudodifferential calculus that for any $a \in C^{\infty}_{0}(B^*_{\tau}),$
\begin{equation} \label{key psdo}
\langle a T_{hol} u_h, T_{hol} u_h \rangle_{L^2(B_{\tau}^*M)} = \langle Op_h(|b_0|^2 a) u_h, u_h \rangle_{L^2(M)} + O(h). \end{equation}  

As indicated in the introduction, we fix the FBI transform in the following and set $T = T_{hol}.$

\section{Localization of the semiclassical wave front of $T_{\Sigma} u_h$} \label{WF}

We first prove a crude bound for $\WF_h(Tu_h)$ and give an explicit example in Section \ref{subsection crude}. Subsequently, we give a more precise bound on $\WF(Tu_h)$ in Section \ref{seclower}. Lastly,   we study $\WF_h(T_\Sigma u_h)$ for a hypersurface $\Sigma$ in Section \ref{secfinal}.

 We remind the reader that in the complexified variables, $(x,\xi)$ are now “space" and the notation $(x^*,\xi^*)$ is used to denote the dual “frequency" variables. Similarly, in the sequel $\alpha=(\alpha_x,\alpha_\xi)$, $\alpha^*=(\alpha_x^*,\alpha_\xi^*)$.

\subsection{A na\"ive bound of the wavefront}\label{subsection crude}

\begin{lemma}[Crude wavefront bounds]\label{lemcrude}
    \begin{equation}\label{crude}
    WF_h(Tu_h)\subset \{(x,\xi,x^*,\xi^*):|\xi|_g=1,\xi^*=0,|x^*|_{\tilde{g}}=1\}.
\end{equation}
\end{lemma}
This will be improved in Proposition \ref{WF1} below. The last condition $|x^*|=1$ in (\ref{crude}) becomes the more precise $x^*=\xi$. 

\begin{proof}
To show (\ref{crude}), notice that the function $Tu_h$ solves
$$
P_{\rho}(h)Tu_h=0,
$$
where $P_{\rho}(h)=e^{-\rho/h}\cdot(-h^2\Delta_{\tilde{g}})\cdot e^{\rho/h}$ is the conjugated K\"ahler Laplacian. This operator was studied in detail in \cite{CT24}, see Section 5 there. To determine its principal symbol, we compute
\begin{align*}
    P_{\rho}f&=-e^{-\rho/h}h^2\Delta (e^{\rho/h}f)\\
    &=-e^{-\rho/h}\big(h^2(\Delta e^{\rho/h})f+2\langle h\nabla e^{\rho/h},h\nabla f\rangle+e^{\rho/h}h^2\Delta f\big),
\end{align*}
in which the first term is 

\begin{align*}
    -e^{-\rho/h}f h^2\Delta e^{\rho/h}
    &=-e^{-\rho/h} f h\text{div}(h\nabla e^{\rho/h})\\
    &=-e^{-\rho/h} f h\text{div}(e^{\rho/h}\nabla \rho)\\
    &=-e^{-\rho/h} f \big(\langle h\nabla e^{\rho/h},\nabla\rho\rangle+e^{\rho/h}h\Delta\rho\big)\\
    &=-e^{-\rho/h} f e^{\rho/h}\langle\nabla\rho,\nabla \rho\rangle+O(h)\\
    &=-|\nabla\rho|^2f+O(h).
\end{align*}
Therefore, the principal symbol is

$$
    \sigma(P_{\rho})=
    |(x^*,\xi^*)|_{\tilde{g}}^2-2i\langle\nabla_{\tilde{g}}\rho,(x^*,\xi^*)\rangle_{\tilde{g}}-|\nabla_{\tilde{g}}\rho|_{\tilde{g}}^2.
$$

By standard wavefront calculus, 

$$
WF_h(Tu_h)\subset \{(x,\xi,x^*,\xi^*)\in T^*(M^{\mathbb{C}}_{\tau}):\sigma(P_{\rho})=0\}=\{\Re\sigma=0\}\cap\{\Im\sigma=0\}.
$$
From the energy localization \cite{GT19} we know that $|\xi|_g^2=1$, thus

$$
\{\Im\sigma=0\}=\{\langle \xi,\xi^*\rangle=0\}=\{\xi^*=0\}.
$$
On the other hand,

$$
\{\Re\sigma=0\}=\{|x^*|^2+|\xi^*|^2=1\}=\{|x^*|^2=1\}.
$$
where we used  $|\nabla\rho|^2=2\rho=|\xi|^2_g=1$.
\end{proof}

The following example makes the comparison transparent. 

\begin{Example}
Take $M=\mathbb{R}^1/\mathbb{Z}^1$ to be the flat circle with $\rho=\xi^2/2$. Then

$$
\sigma(P_{\rho})=0\Longleftrightarrow
\begin{cases}
(x^*)^2+(\xi^*)^2-\xi^2=0\\
\xi \xi^*=0\\    
\end{cases}
$$

Direct wavefront calculus gives
\begin{equation}\label{4 circles}
WF_{h}(Tu_h)\subset\{(x,\xi,x^*,\xi^*):x\in \mathbb{R}/\mathbb{Z},~\xi=\pm1,~x^*=\pm1,~\xi^*=0\}.    
\end{equation}
The right-hand side consists of four circles.

Proposition \ref{WF1} in the next section will give
\begin{equation}\label{2 circles}
WF_{h}(Tu_h)\subset\{(x,\xi,x^*,\xi^*):x\in \mathbb{R}/\mathbb{Z},~\xi=x^*,~x^*=\pm1,~\xi^*=0\}.    
\end{equation}
The right-hand side consists of two circles.

One can show that the two sides in (\ref{2 circles}) are actually equal. To see this, consider subsequences of the $L^2$ orthonormal basis on the flat circle, $u_{1/k}(x)=e^{-ikx}$ and $v_{1/k}(x)=e^{+ikx}$. Then $u^{\mathbb{C}}_{1/k}(z)=e^{-ik(x-i\xi)}$. Note that $z$ is identified with $x-i\xi$ not $x+i\xi$, by the construction in the previous section. Now compute

$$
Tu_{1/k}(x,\xi)=k^{1/4}e^{-k/2}e^{-k\xi^2/2}e^{-ik(x-i\xi)}=k^{1/4}e^{-\frac{k}{2}(\xi+1)^2}e^{-ikx}.
$$

Clearly, this sequence concentrates near $x\in \mathbb{R}/\mathbb{Z},~ \xi=-1$. We will apply the unitary semiclassical Fourier transform $(x,
\xi)\to(x^*,\xi^*)$ and see that its frequency variables concentrate near $x^*=-1,~\xi^*=0$.

\begin{equation*}
    \begin{split}
        \mathcal{F}_h(Tu_h)(x^*,\xi^*)
        &=\frac{h^{-1/4}}{2\pi h}\int_{\mathbb{R}^2}e^{-\frac{i}{h}(xx^*+\xi\xi^*)}e^{-\frac{(\xi+1)^2}{2h}}e^{-\frac{i}{h}x}dxd\xi\\
        &=\frac{h^{-1/4}}{2\pi h}\int_{\mathbb{R}}e^{-\frac{i}{h}xx^*}e^{-\frac{i}{h}x}dx\int_{\mathbb{R}}e^{-\frac{i}{h}\xi\xi^*}e^{-\frac{(\xi+1)^2}{2h}}d\xi\\
        &=h^{-1/4}\mathcal{F}_h(1)(x^*+1)\cdot\mathcal{F}_h(e^{-\frac{(\xi+1)^2}{2h}})(\xi^*)\\
        &=\sqrt{2\pi}h^{1/4}\delta_{-1}(x^*)e^{-\frac{(\xi^*)^2}{2h}}e^{\frac{i}{h}\xi^*}
        \end{split}
\end{equation*}
in the sense of oscillatory integrals. As a result, the weak* limit is

$$
\mathcal{F}_h(Tu_h)\xrightarrow[w^*]{h\to 0^+}\delta_{x^*=-1}\otimes\delta_{\xi^*=0}.
$$

That shows $WF_h(Tu_h)=\{(x,-1,-1,0):x\in\R/\Z\}$, thus this sequence fills in one of the limit circles in (\ref{2 circles}).

Similarly, the other sequence $v_{1/k}$ with $(\xi-1)^2$ replacing $(\xi+1)^2$ fills the other limit circle with $\xi=1,~x^*=1$. These two sequences of trigonometric functions, together with constant mode, form the complete $L^2$ orthonormal basis on the flat circle.
\end{Example}

As a result, the bounds in Lemma \ref{lemcrude} is not sharp. Proposition \ref{WF1} below is sharp in the sense that it is saturated in the case of the circle.

\subsection{A sharp bound of $WF_h(Tu_h)$}\label{seclower}

\begin{proposition}\label{WF1}
Let $(M,g)$ be a compact $C^{\omega}$ Riemannian manifold and $\{ u_h \}$ be {\em any} sequence of $L^2$-normalized Laplace eigenfunctions on $M.$ Then, the semiclassical wavefront set of $Tu_h$ satisfies
    \begin{equation}\label{wf1}
        \begin{split}
            WF_h(Tu_h)\subset\big\{(x,\xi :&x^*,\xi^*)\in T^*(B_{\tau}^*M): |\xi|_g=1, \, \xi^* = 0, \, x^* = \xi \big\}.
            \end{split}
    \end{equation}
 In particular, $WF_h(Tu_h)$ is a  compact subset of $T^* (B_{\tau}^*M).$
\end{proposition}

\begin{proof}

Since $u_h$  are Laplace eigenfunctions on $M$, they $h$-microlocally concentrate near the cosphere bundle $\{|\xi|_g^2=1\}\subset M^{\mathbb{C}}_{\tau}$. In particular, in the real analytic case  (see \cite{GT19} Prop. 2.3) given any cutoff $ \chi_{\epsilon} \in C^{\infty}_0 (\R)$ with supp $\chi_{\epsilon} \subset [-2\epsilon, 2\epsilon] $ and $\chi_{\epsilon} |_{[-\epsilon, \epsilon]} = 1,$ for any $k \in \N,$

$$ \| (1- \chi_{\epsilon}) ( |\xi|_g -1) \, T u_h \|_{C^{k}} = O_{k} (e^{-C_{\epsilon}/h}), \quad C_{\epsilon} >0.$$
In the smooth case, the exponential decay $O(e^{-C/h})$ is replaced with $O(h^{\infty}).$
In particular, it follows that $$
WF_h(Tu_h)\subset \{(x,\xi,x^*,\xi^*)\in T^*(M ^{\mathbb{C}}_{\tau}):|\xi|_g=1\}.
$$

In the following, we set $\chi_{\epsilon}^+:= (1-\chi_{\epsilon}).$
To further $h$-microlocalize near $x^*= \xi$ and $\xi^*=0,$ we use a reproducing formula for $Tu_h$.  We would like to express the function $Tu_h$ as an average of the value of itself. 
Recall that given FBI transform $T$ with
$$
Tu_h(\beta)=\int_Me^{i\varphi(\beta,y)/h}a(\beta,y;h)\chi(\beta_x,y)u_h(y)dy.
$$
by the left parametrix construction in (\ref{leftparametrix}), there exists some $b\in S_{cla}^{\frac{3n}{4},\frac{n}{4}}$, 
$$
Sv_h(x)=\int_{T^*M}e^{-i\varphi^*(\alpha,x)/h}b(\alpha,x;h)\chi(\alpha_x,x)v_h(\alpha)d\alpha.
$$
with $STu_h=u_h+Ru_h,\quad R\in O(e^{-C/h})$. 

We use the reproducing formula
\begin{equation} \label{reproduce}
T u_h =  T S Tu_h + O(e^{-C/h}).
\end{equation}

Writing (\ref{reproduce}) out explicitly, we have

\begin{align} \label{reproduce2}
Tu_h(\beta) = \int_M \int_{T^*M}e^{\frac{i}{h}[\varphi(\beta,x) -\varphi^*(\tilde{\beta}, x)]}c(\beta,\tilde{\beta},x) Tu_h(\tilde{\beta})d\tilde{\beta} dx  +O(e^{-C/h}), \nonumber \\
 c(\beta,\tilde{\beta},x):= a(\beta,x)b(\tilde{\beta},x) \chi(\beta_x,x) \chi(\tilde{\beta}_x,x) \hspace{1in}
\end{align}

In the following, we denote the total phase in (\ref{reproduce2}) by
\begin{equation}\label{Phi phase}
\Phi(\beta,\tilde{\beta},x):= \varphi(\beta,x)-\varphi^*(\tilde{\beta},x).
\end{equation}
Then, using (\ref{reproduce}) we note that writing $\alpha = (\alpha_x,\alpha_{\xi}), \, \beta = (\beta_x, \beta_{\xi}),$
\begin{align} \label{cutoff1}
\chi_{\epsilon}^+( hD_{\beta_x} - \beta_{\xi}) Tu(\alpha) =  h^{-2n} \int \cdots \int e^{i  \big( \langle \alpha - \beta, \beta^* \rangle  + \Phi(\beta,\tilde{\beta},x)  \big)/h } \, \chi_{\epsilon}^+ (\beta_x^*- \beta_{\xi})  \nonumber \\
\times c(\beta,\tilde{\beta},x) Tu_h(\tilde{\beta})d\beta d\tilde{\beta} dx d\beta^*+O(h^{\infty}).
\end{align}

Since

$$ \partial_{\beta_x} \big( \langle \alpha - \beta, \beta^* \rangle + \Phi(\beta,\tilde{\beta},x) \big)  = - \beta_x^* + \beta_{\xi} +  i E(\beta,\tilde{\beta},x),$$
where $|E(\beta,\tilde{\beta},x)| \leq C |\beta_x-x|$ and $\Re E = 0.$  We  then decompose (\ref{cutoff1}) further to  the sets where $C |\beta_x-x| < \ep/2$ and $ C |\beta_x-x| > \ep/2$ respectively. More precisely, we write
\begin{align} \label{cutoff2}
h^{2n}\chi_{\epsilon}^+( hD_{\beta_x} - \beta_{\xi}) Tu(\alpha) \hspace{4in} \nonumber \\
=\int \cdots \int e^{i  \big( \langle \alpha - \beta, \beta^* \rangle  + \Phi(\beta, \tilde{\beta},x)  \big)/h } \, \chi_{\epsilon}^+ (\beta_x^*- \beta_{\xi})  \chi_{\ep/2}^+ (\beta_x -x) c Tu_h(\tilde{\beta})d\beta d{\tilde{\beta}} dx  d\beta^* \nonumber \\
+   \int \cdots \int e^{i  \big( \langle \alpha - \beta, \beta^* \rangle  + \Phi(\beta,\tilde{\beta},x)  \big)/h } \, \chi_{\epsilon}^+ (\beta_x^*- \beta_{\xi})  \chi_{\ep/2}(\beta_x -x) c Tu_h(\tilde{\beta})d\beta d\tilde{\beta} dx d\beta^*
\end{align}

To bound the first integral in (\ref{cutoff2}) we note that 
$$ \Im \Phi(\beta,\tilde{\beta},x) \geq C |\beta_x-x|^2 \geq C \epsilon^2$$
when $|\beta_x -x| \geq \ep/2$ and so, this integral is clearly $O(e^{-C \ep^2/h}).$ 

As for the second integral, we note that when $C |\beta_x -x| < \ep/2$ and $|\beta_{\xi} - \beta_x^*| > \ep,$
$$ |  \beta_x^* - \beta_{\xi} + E(\beta,x) | \geq \ep - \frac{\ep}{2} \geq \frac{\ep}{2},$$
and one can integrate by parts with respect to $ L_1 =  \frac{ \overline{\partial_{\beta_x} ( \langle \alpha - \beta, \beta^* \rangle  + \Phi)} \cdot h D_{\beta_x}}{  | \partial_{\beta_x} ( \langle \alpha - \beta, \beta^* \rangle  + \Phi)  |^2}$ using that 
$$ L_1 ( e^{i ( \langle \alpha - \beta, \alpha^* \rangle  + \Phi )/h} ) = e^{i ( \langle \alpha - \beta, \alpha^* \rangle  + \Phi )/h}$$
and the fact that the denominator 
$$ | \partial_{\beta_x} ( \langle \alpha - \beta, \beta^* \rangle  + \Phi)  |^2 \geq \frac{\ep^2}{4}$$
on the support of the integrand. The result is that
$$
h^{-2n} \int \cdots \int e^{i  \big( \langle \alpha - \beta, \beta^* \rangle  + \Phi  \big)/h } \, \chi_{\epsilon}^+ (\beta_x^*- \beta_{\xi})  \chi_{\ep/2}(\beta_x -x) c Tu_h(\tilde{\beta})d\beta d\tilde{\beta} dx d\beta^* = O(h^{\infty})
$$

Consequently, it follows that
\begin{equation}\label{wf3}
 \text{WF}_h(Tu_h)\subset \big\{(x,\xi :x^*,\xi^*)\in T^*(B_{\tau}^*M): |\xi|_g=1, \, \, x^* = \xi \big\}.
\end{equation}

To complete the proof, we note that
$$
\partial_{\beta_{\xi}} \big( \langle \alpha - \beta, \beta^* \rangle + \Phi(\beta,\tilde{\beta},x) \big)  = \beta_x - x  - \beta_{\xi}^*  + i \tilde{E}(\beta,\tilde{\beta},x),
$$
where $\tilde{E}(\beta,\tilde{\beta},x) = O(|\beta_x -x|^2)$ and $\Re \tilde{E} = 0.$ By the same argument as above, one can then write

\begin{align} \label{wf4}
h^{2n} \chi_{\epsilon}^+( hD_{\beta_{\xi}} ) Tu(\alpha) \hspace{4in} \nonumber \\
= \int \cdots \int e^{i  \big( \langle \alpha - \beta, \beta^* \rangle  + \Phi(\beta,\tilde{\beta},x)  \big)/h } \, \chi_{\epsilon}^+ (\beta_{\xi}^*)  \chi_{\ep/2}^+ (x-\beta_x ) c Tu_h(\tilde{\beta})d\beta d\tilde{\beta} dx d\beta^* \nonumber \\
+ \int \cdots \int e^{i  \big( \langle \alpha - \beta, \beta^* \rangle  + \Phi(\beta,\tilde{\beta},x)  \big)/h } \, \chi_{\epsilon}^+ (\beta_{\xi}^*)  \chi_{\ep/2}(x-\beta_x ) cTu_h(\tilde{\beta})d\beta d\tilde{\beta} dx d\beta^*
\end{align}

As in (\ref{wf3}), for the first term on the RHS of (\ref{wf4}), we note that when $|x-\beta_x | \geq \ep,$
$\Im \Phi \geq C \epsilon^2$ and so the first integral is $O(e^{-C \epsilon^2/h}).$
As for the second integral, when $|x-\beta_x | \leq \frac{\ep}{2}$ and $|\beta_{\xi}^*| >\epsilon,$ it follows that after possibly shrinking $\epsilon>0$ further,
$$ | \partial_{\beta_{\xi}} \big( \langle \alpha - \beta, \beta^* \rangle + \Phi(\beta,\tilde{\beta},x) \big)|  \geq \frac{\ep}{2}.$$
One can then integrate by parts with respect to 

 $ L_2=  \frac{ \overline{\partial_{\beta_{\xi}} ( \langle \alpha - \beta, \beta^* \rangle  + \Phi)} \cdot h D_{\beta_{\xi}}}{  | \partial_{\beta_{\xi}} ( \langle \alpha - \beta, \beta^* \rangle  + \Phi)  |^2}$ using that 
$$ L_2 ( e^{i( \langle \alpha - \beta, \beta^* \rangle  + \Phi)/h} ) = e^{i ( \langle \alpha - \beta, \beta^* \rangle  + \Phi)/h}$$
and the fact that the denominator 
$$ | \partial_{\beta_{\xi}} ( \langle \alpha - \beta, \beta^* \rangle  + \Phi)  |^2 \geq \frac{\ep^2}{4}$$
on the support of the integrand. The result is that

$$ \int \cdots \int e^{i  \big( \langle \alpha - \beta, \beta^* \rangle  + \Phi  \big)/h } \, \chi_{\epsilon}^+ (\beta_{\xi^*})  \chi_{\ep/2}(x-\beta_x) c Tu_h(\tilde{\beta})d\beta d\tilde{\beta} dx d\beta^* = O(h^{\infty}).$$

Since in the above $\ep>0$ is fixed arbitrarily small, in view of (\ref{wf3}), this completes the proof of the Proposition.

\end{proof}

In the following we set
\begin{equation} \label{W}
{\mathcal W}:= \big\{(x,\xi :x^*,\xi^*)\in T^*(B_{\tau}^*M): |\xi|_g=1, \, \xi^* = 0, \, x^* = \xi \big\}.
\end{equation}

Since ${\mathcal W} \subset T^* (B_\tau^*M)$ is compact, by $C^{\infty}$ Urysohn lemma we let $\chi_{{\mathcal W}}\in C^{\infty}_0(T^*B^*_{\tau}M)$ with $\chi_{{\mathcal W}} (x,\xi,x^*,\xi^*) = 1 $ for $(x,\xi,x^*,\xi^*)$ in a  collar neighbourhood of ${\mathcal W}$ in $T^*B^*_{\tau}M.$

\subsection{The wavefront set $WF_h(T_{\Sigma} u)$}\label{secfinal}
In the proof of Theorem \ref{thm1}, we will need to show localization (i.e. compactness) of the restricted wavefront $WF_h (T_{\Sigma} u_h).$ In this subsection, we do this by deriving the relation between $WF_h(Tu_h)$ and $WF_h(T_{\Sigma}u_h).$

Before stating this result, we review some background on the symplectic geometry. Let $\Sigma = \{ (x,\xi) \in B_\tau^*M;  F(x,\xi) = 0 \}$ where $dF |_{\Sigma} \neq 0.$ Considering $F$ as a function on $T^* (B_\tau^*M)$ (constant in the fiber coordinates $(x^*,\xi^*)$) it follows that the Hamilton vector field of $F$ with respect to the canonical symplectic form $\Omega = dx \wedge dx^* + d\xi \wedge d\xi^*$ on $T^* B_\tau^*M$ is just
$$ - X_F =  \partial_x F  \, \partial_{x^*} + \partial_{\xi}F \, \partial_{\xi^*}$$
and the associated Hamilton flow is given by
\begin{equation} \label{flow}
\exp \, \tau X_F(x,\xi,x^*,\xi^*) = (x,\xi; x^*,\xi^*)  - \tau \, (0,0, \partial_x F,  \partial_{\xi}F ).
\end{equation}
We also recall that given a closed hypersurface $\Sigma \subset B^*_{\tau}M$, there is a natural projection map $\pi_{\Sigma}: T^*_{\Sigma} B^*_{\tau}M \to T^*\Sigma$. Indeed if $(\beta',0;\beta'^*,\beta^*_{2n})$ denote Fermi coordinates in a neighbourhood of $\Sigma$ with $\Sigma = \{ \beta_{2n} = 0 \}$, then the projection map is $\pi_{\Sigma}(\beta',0; \beta'^*,\beta_{2n}^*) = (\beta',\beta'^*)$.

\begin{proposition} \label{WF2}
Let $(M,g)$ be a compact $C^{\omega}$ Riemannian manifold, $\{ u_h \}$ be {\em any} sequence of $L^2$-normalized Laplace eigenfunctions on $M$ and $\Sigma=\{ F = 0 \}$ with $dF |_{\Sigma} \neq 0,$ be a compact hypersurface in the Grauert tube $B_{\tau}^*M.$ Then,    there exists a constant $T_F>0$ such that  \begin{equation}\label{wf1'}
        WF_h(T_{\Sigma}u_h)\subset \bigcup_{|\tau| \leq T_F} \exp \tau X_{F} \, \pi_{\Sigma}({\mathcal W}),
        \end{equation}
    where ${\mathcal W}$ is defined in (\ref{W}).
In particular, $WF_h(T_{\Sigma}u_h)$ is a  compact subset of $T^*\Sigma$.
\end{proposition}

\begin{proof}Just as in section \ref{seclower}, the starting point is the reproducing formula (\ref{reproduce}). Given a compact hypersurface $\Sigma \subset B_{\tau}^*M,$ restriction of (\ref{reproduce}) to $\Sigma$ gives
\begin{equation} \label{wfr1}
T_\Sigma u_h =  T_{\Sigma}S T u_h + O(e^{-C/h}).
\end{equation}

Let $\Sigma = \{ F(x,\xi) = 0 \}$ where $dF |_{\Sigma} \neq 0.$  Given a point $q_0  \in \Sigma$ it follows by possibly reordering coordinates that either $\partial_{\xi_n} F(q_0) \neq 0$ or $\partial_{x_n} F (q_0) \neq 0.$

Assume first that $\partial_{\xi_n} F (x,\xi) \neq 0$ for all $(x,\xi) \in \Sigma \cap U$ near $q_0$, so that $F(x,\xi)=\xi_n-G(x,\xi')$.Then, by the implicit function theorem $ \Sigma \cap U = \{ (x, \xi', \xi_n = G(x,\xi') \}$ where $G \in C^{\infty}_{loc}.$ Consequently, one can use $\alpha_{\Sigma} := (\alpha_x, \alpha_{\xi'})$ as local coordinates on $\Sigma \cap U$ and we denote canonical dual coordinates by $\alpha_{\Sigma}^*:= (\alpha_x^*,\alpha_{\xi'}^*).$ It follows that

\begin{align} \label{wfr2}
Op_{h,\Sigma}(c)  T_{\Sigma} u_h (\alpha_{\Sigma})  \hspace{4in} \nonumber \\
= \int e^{i \Psi_{\Sigma}(\alpha_{\Sigma}, \beta_{\Sigma}, \beta_{\Sigma}^*, \tilde{\beta},x)/h} \, c(\alpha_{\Sigma}, \beta_{\Sigma}^*) a (\beta_{\Sigma}, \tilde{\beta},x) Tu_h(\tilde{\beta}) d\tilde{\beta}  dx d\beta_{\Sigma} d \beta_{\Sigma}^*  + O(e^{-C/h}),
\end{align}
where the total phase is
\begin{equation} \label{wfr3}
\Psi_{\Sigma}(\alpha_{\Sigma}, \beta_{\Sigma},\beta_\Sigma^*,\tilde{\beta},x) = \langle \alpha_{\Sigma} - \beta_{\Sigma}, \beta_{\Sigma}^* \rangle + \Phi( \beta_\Sigma, G(\beta_{\Sigma}); \tilde{\beta},x),
\end{equation}
where $\Phi$ was given in (\ref{Phi phase}).

The critical point equations in the $\beta_{\Sigma} = (\beta_x, \beta_{\xi'})$ variables are
\begin{eqnarray} \label{wfr4}
\partial_{\beta_x} \Psi_{\Sigma} = - \beta_{x}^* + \partial_{\beta_x} \Phi + \partial_{\beta_{\xi_n}} \Phi \cdot \partial_{\beta_x} G = 0, \nonumber \\
\partial_{\beta_{\xi'}} \Psi_{\Sigma} = - \beta_{\xi'}^* + \partial_{\beta_{\xi'}} \Phi  + \partial_{\beta_{\xi_n}} \Phi \cdot \partial_{\beta_{\xi'}} G = 0,
\end{eqnarray}

which give 

\begin{eqnarray} \label{crit}
 - \beta_{x'}^* + \beta_{\xi'}  = O(|\beta_{x} -x|),  \,\,\,\, -\beta_{x_n}^* + G(\beta_x,\beta_{\xi'}) = O(|\beta_x -x|), \,\,\,\,
  \beta_{\xi'}^* = O(|\beta_x -x|). 
\end{eqnarray}

In (\ref{crit}) we have used that $\partial_{\beta_{\xi}} \Phi = O(|\beta_x -x|)$ and $\partial_{\beta_x} \Phi = \beta_{\xi}  + O(|\beta_x -x|).$ Then, using the fact that $\Im \Phi \geq C |\beta_x -x|^2 $ it follows by a similar integration-by-parts argument as in the proof of Proposition \ref{WF1} that for any $\ep >0,$ and with $G = G(x,\xi'),$

\begin{eqnarray} \label{wfr5}
\chi_{\ep}^{+}( hD_{\beta_{x'}} - \beta_{\xi'}, h D_{\beta_{x_n}} - G) = O_{L^2 \to L^2}(h^{\infty}), \nonumber \\
\chi_{\ep}^{+}( hD_{\beta_{\xi'} }) = O_{L^2 \to L^2}(h^{\infty}).
\end{eqnarray}

It then follows from (\ref{wfr5}) and eigenfunction energy concentration that locally near $q_0 \in \Sigma$ with $\partial_{\xi_n} F(q_0) \neq 0,$

\begin{eqnarray} \label{wfr6}
\text{WF}_h (T_{\Sigma} u |_{U} ) \subset \pi_{\Sigma}({\mathcal W}), \end{eqnarray}
where
$$ \pi_{\Sigma}({\mathcal W}) = \big\{ (x,\xi', x^*, \xi'^*) \in T^* U ; \,\, x'^* = \xi', \, x_n^* = G(x,\xi'), \xi'^* = 0, \, |(\xi', G) |_{g(x)} = 1  \big\}.$$

Next, we consider the case where $q_0 \in \Sigma $ with $\partial_{x_n} F (q_0) \neq 0.$ Then, again by the implicit function theorem, there exists a neighbourhood $B_\tau^* M \supset V \ni q_0$  with 
$$ \Sigma \cap V = \{ (x',\xi); x_n = H(x',\xi), \, (x,\xi) \in V \}, \quad H \in C_{loc}^{\infty}(\R^{2n-1}).$$
Thus, in this case, we can fix the defining function $F(x,\xi) = x_n - H(x',\xi).$

We use $\alpha_{\Sigma}:= (\alpha_{x'},\alpha_\xi)$ as local  hypersurface  coordinates on $\Sigma \cap V$ and denote the canonical dual coordinates by $\alpha_{\Sigma}^* = (\alpha_{x'}^*,\alpha_\xi^*).$

\begin{align} \label{wfr7}
Op_{h,\Sigma}(c)  T_{\Sigma} u_h (\alpha_{\Sigma})  \hspace{4in} \nonumber \\
= \int e^{i \Psi_{\Sigma}(\alpha_{\Sigma}, \beta_{\Sigma}, \beta_{\Sigma}^*, \tilde{\beta},x)/h} \, c(\alpha_{\Sigma}, \beta_{\Sigma}^*) a (\beta_{\Sigma}, \tilde{\beta},x) Tu_h(\tilde{\beta})   d\tilde{\beta}  dx d\beta_{\Sigma} d \beta_{\Sigma}^*   + O(e^{-C/h}),
\end{align}
where the total phase
\begin{equation} \label{wfr3'}
\Psi_{\Sigma}(\alpha_{\Sigma}, \beta_{\Sigma},\beta_\Sigma^*,\tilde{\beta},x) = \langle \alpha_{\Sigma} - \beta_{\Sigma}, \beta_{\Sigma}^* \rangle + \Phi( \beta_{x'}, H(\beta_{x'},\beta_\xi), \beta_{\xi}; \tilde{\beta},x).
\end{equation}

The critical point equations in $\beta_{\Sigma} = (\beta_{x'},\beta_\xi)$ are

\begin{eqnarray} \label{wfr8}
\partial_{\beta_{x'}} \Psi_{\Sigma} = - \beta_{x'}^* + \partial_{\beta_{x'}} \Phi + \partial_{\beta_{x_n}} \Phi \cdot \partial_{\beta_{x'}} H = 0, \nonumber \\
\partial_{\beta_{\xi}} \Psi_{\Sigma} = - \beta_{\xi}^* + \partial_{\beta_{\xi}} \Phi  + \partial_{\beta_{x_n}} \Phi \cdot \partial_{\beta_{\xi}} H = 0,
\end{eqnarray}
which give

\begin{eqnarray} \label{crit2}
 - \beta_{x'}^* + \beta_{\xi'}  + \beta_{\xi_n} \, \partial_{\beta_{x'}} H  = O(|\beta_{x} -x|), \nonumber \\
  - \beta_{\xi}^* + \beta_{\xi_n} \, \partial_{\beta_{\xi}} H  = O(|\beta_x -x|). 
\end{eqnarray}

Then, using the fact that $\Im \Phi \geq C |\beta_x -x|^2 $ it follows by a similar integration-by-part argument as in the proof of Proposition \ref{WF1} that for any $\ep >0,$ and with $H = H(x',\xi),$

\begin{eqnarray} \label{wfr9}
\chi_{\ep}^{+}( hD_{\beta_{x'}} - \beta_{\xi'} - \beta_{\xi_n} \partial_{\beta_{x'}}H) = O_{L^2 \to L^2}(h^{\infty}), \nonumber \\
\chi_{\ep}^{+}( hD_{\beta_{\xi} } - \beta_{\xi_n} \partial_{\beta_{\xi}} H) = O_{L^2 \to L^2}(h^{\infty}).
\end{eqnarray}

It then follows from (\ref{wfr9}) and eigenfunction energy concentration that locally near $q_0 \in \Sigma$ with $\partial_{x_n} F(q_0) \neq 0,$

\begin{equation} \label{wfr10}
\text{WF}_h (T_{\Sigma} u |_{V} ) \subset \big\{ (x',\xi, x'^*, \xi^*) \in T^* V ; \,\, x'^* = \xi' +  \, \xi_n \, \partial_{x'} H, \,  \xi^* =  \xi_n \, \partial_{\xi}H, \, |\xi |_{g(x',H)} = 1  \big\}.
\end{equation}

From (\ref{wfr6}), it follows that
$$ \text{WF}_h (T_{\Sigma} u |_{U} ) \subset \pi_{\Sigma} |_{U}({\mathcal W})$$
By the  energy localization condition $|\xi|_{g(x',H)}=1,$ we have that $|\xi_n| \leq T_{H}$ for some constant $T_H>0.$ From (\ref{flow}) and (\ref{wfr10}) it then follows that

 $$\text{WF}_h (T_{\Sigma} u |_{V} ) \subset  \bigcup_{|\tau| \leq T_H} \exp \tau X_H \, \pi_{\Sigma}|_{V} {\mathcal W}.$$ 

Since $\Sigma$ can be covered by finitely many open sets of the form $U$ and $V$ the proposition follows.
\end{proof}

In the following we set 

\begin{equation} \label{W2}
{\mathcal W}_{\Sigma} :=  \bigcup_{|\tau| \leq T_F} \exp \tau X_F \, \pi_{\Sigma} ({\mathcal W}).
\end{equation}

Since ${\mathcal W}_{\Sigma} \subset T^* \Sigma $ is compact, by $C^{\infty}$ Urysohn lemma we let $\chi_{\Sigma} \in C^{\infty}_0(T^*\Sigma)$ with $\chi_{\Sigma} (x,\xi,x^*,\xi^*) = 1 $ for $(x,\xi,x^*,\xi^*)$ in a fixed Fermi neighbourhood of ${\mathcal W}_\Sigma$ in $T^*\Sigma.$

\begin{remark}
    Since $ \big( \overline{\partial}_{\beta} + i \beta_{\xi} \big) \Phi(\beta,x)  = 0$, one gets $\big( \overline{\partial_b}_{\beta} + i \beta_{\xi} ) \Phi_{\Sigma}(\beta,x) = 0$ where $\beta \in \Sigma$ and $\overline{\partial_b }$ is the induced tangential CR vector field on $\Sigma$. 
\end{remark} 
\begin{remark}
    Since $\mathcal{W}_{\Sigma}$ is a flow-out of $\pi_{\Sigma}(\mathcal{W})$ along the fiber directions $(x^*,\xi^*)$, we have that the 2-microlocal $h$-singular support of $T_{\Sigma}u_h$ satisfies
    \begin{equation}\label{singsupp}
    \mathrm{sing~supp}_h(T_{\Sigma}u_h)\subset p_{\Sigma}(\mathcal{W}_{\Sigma})\subset S^*M\cap \Sigma
    \end{equation}
    where $p_{\Sigma}:T^*(\Sigma)\to \Sigma$ is the canonical projection.
\end{remark}
\begin{remark} In the following, it will also be useful to note that as a consequence of the proof of Proposition \ref{WF2}, 
\begin{equation} \label{refinement}
WF_{h}(T_{\Sigma} u_h ) \subset \{ (\alpha_{\Sigma}, \alpha_{\Sigma}^{*}) \in T^*\Sigma; \, \alpha_{\Sigma}^* = \tilde{F}(\alpha_{\Sigma}) \} =: {\mathcal W}_{\Sigma}',
\end{equation}
where $\tilde{F} \in C^{\infty}(\Sigma, \R^{2n-1})$ depends on the defining function $F.$
\end{remark}

The following result is of independent interest. Roughly speaking, it says that to leading order any $Q \in \Psi_h^{0}(\Sigma)$ acting on $T_{\Sigma} u_h$ can be written as a multiplication operator.

\begin{proposition} \label{multiplier}
Let $\Sigma \subset B_{\tau}^*M$ be a real, oriented,  compact hypersurface. Then, given $Q(h) \in \Psi_h^{0}(\Sigma),$ there exists a {\em function} $q_{\Sigma} \in C^{\infty}(\Sigma)$ such that
$$ Q(h) T_{\Sigma} u_h = q_{\Sigma} \,T_{\Sigma} u_h + o(1) \| T_{\Sigma} u_h \|_{L^2},$$
with $q_{\Sigma} = \sigma(Q) |_{{\mathcal W}_{\Sigma}'}.$ 
\end{proposition}

\begin{proof}
Following the argument in Proposition \ref{WF2}, we
assume first that $\partial_{\xi_n} F (x,\xi) \neq 0$ for all $(x,\xi) \in \Sigma \cap U$ near $q_0.$ Then, by the implicit function theorem, $ \Sigma \cap U = \{ (x, \xi'); \xi_n = G(x,\xi') \}$ where $G \in C^{\infty}_{loc}.$ Consequently, one can use $\alpha_{\Sigma} := (\alpha_x, \alpha_{\xi'})$ as local coordinates on $\Sigma \cap U$ and we denote canonical dual coordinates by $\alpha_{\Sigma}^*:= (\alpha_x^*,\alpha_{\xi'}^*).$ It follows that

\begin{align} \label{mult1}
Q(h) T_{\Sigma} u_h (\alpha_{\Sigma})  \hspace{4in} \nonumber \\
= \int e^{i \Psi_{\Sigma}(\alpha_{\Sigma}, \beta_{\Sigma}, \beta_{\Sigma}^*, \tilde{\beta},x)/h} \, q(\alpha_{\Sigma}, \beta_{\Sigma}^*) a (\beta_{\Sigma}, \tilde{\beta},x) Tu_h(\tilde{\beta}) d\tilde{\beta}  dx d\beta_{\Sigma} d \beta_{\Sigma}^*  + O(e^{-C/h}),
\end{align}
where the total phase
\begin{equation} \label{mult2}
\Psi_{\Sigma}(\alpha_{\Sigma}, \beta_{\Sigma},\beta_{\Sigma}^*, \tilde{\beta},x) = \langle \alpha_{\Sigma} - \beta_{\Sigma}, \beta_{\Sigma}^* \rangle + \Phi( \beta_\Sigma, G(\beta_{\Sigma}); \tilde{\beta},x).
\end{equation}

In view of the critical point equations in (\ref{crit}), one Taylor expands the symbol $q(\alpha_{\Sigma}, \beta_{\Sigma}^*)$ around $\beta_{x'}^* = \beta_{\xi'}, \, \beta_{x_n}^* = G(\beta_x, \beta_{\xi'})$ and $\beta_{\xi'}^* = 0.$ Setting $\beta_{\Sigma}^{0} = (  \beta_{\xi'},  G(\beta_x, \beta_{\xi'}), 0)$ and writing

$$ 
q(\alpha_{\Sigma}, \beta_{\Sigma}^*) = q(\alpha_{\Sigma}, \beta_{\Sigma}^{0}) + A ( \beta^{*}_{\Sigma} - \beta_{\Sigma}^{0}  ), \quad A = (A_1,A_2,A_3) \text{ with} \, A_j \in S^{0}(\Sigma); j=1,2,3,
$$

one writes the integral on the RHS of (\ref{mult1}) in the form

\begin{align} \label{mult3}
 \int e^{i \Psi_{\Sigma}(\alpha_{\Sigma}, \beta_{\Sigma}, \beta_{\Sigma}^*, \tilde{\beta},x)/h} \, q(\alpha_{\Sigma}, \beta_{\Sigma}^{0}) \, a (\beta_{\Sigma}, \tilde{\beta},x) Tu_h(\tilde{\beta}) d\tilde{\beta}  dx d\beta_{\Sigma} d \beta_{\Sigma}^* \nonumber \\
 + \int e^{i \Psi_{\Sigma}(\alpha_{\Sigma}, \beta_{\Sigma}, \beta_{\Sigma}^*, \tilde{\beta},x)/h} \, A \cdot ( \beta^{*}_{\Sigma} - \beta_{\Sigma}^{0}  )  \, a (\beta_{\Sigma}, \tilde{\beta},x) Tu_h(\tilde{\beta}) d\tilde{\beta}  dx d\beta_{\Sigma} d \beta_{\Sigma}^*=: I_1 + I_2
\end{align}

For the first term $I_1$ we make a standard Taylor expansion of the amplitude around $\alpha_{\Sigma} = \beta_{\Sigma}$ and integrate by parts with respect to $hD_{\beta_{\Sigma}^*}$ to get that

\begin{eqnarray} \label{I1}
I_1 =   \int e^{i \Psi_{\Sigma}(\alpha_{\Sigma}, \beta_{\Sigma}, \beta_{\Sigma}^*, \tilde{\beta},x)/h} \, q(\alpha_{\Sigma}, \alpha_{\Sigma}^{0}) \, a (\beta_{\Sigma}, \tilde{\beta},x) Tu_h(\tilde{\beta}) d\tilde{\beta}  dx d\beta_{\Sigma} d \beta_{\Sigma}^* + O(h) \| T_{\Sigma} u \|_{L^2} \nonumber \\
=  q(\alpha_{\Sigma}, \alpha_{\Sigma}^{0}) T_{\Sigma}u +  O(h) \| T_{\Sigma} u \|_{L^2}. \hspace{2in}
 \end{eqnarray}

 To estimate $I_2$ we make a further decomposition and write
\begin{eqnarray} \label{I2}
I_2 = \int e^{i \Psi_{\Sigma}(\alpha_{\Sigma}, \beta_{\Sigma}, \beta_{\Sigma}^*, \tilde{\beta},x)/h} \, \langle A, \beta^{*}_{\Sigma} - \beta_{\Sigma}^{0}  \rangle  \, \chi_{\epsilon} ( \beta^{*}_{\Sigma} - \beta_{\Sigma}^{0}  ) \, a (\beta_{\Sigma}, \tilde{\beta},x) Tu_h(\tilde{\beta}) d\tilde{\beta}  dx d\beta_{\Sigma} d \beta_{\Sigma}^* \nonumber \\
+ \int e^{i \Psi_{\Sigma}(\alpha_{\Sigma}, \beta_{\Sigma}, \beta_{\Sigma}^*, \tilde{\beta},x)/h} \, \langle A,   \,  \beta^{*}_{\Sigma} - \beta_{\Sigma}^{0}  \rangle  \, \chi_{\epsilon}^+ ( \beta^{*}_{\Sigma} - \beta_{\Sigma}^{0}  ) \, a (\beta_{\Sigma}, \tilde{\beta},x) Tu_h(\tilde{\beta}) d\tilde{\beta}  dx d\beta_{\Sigma} d \beta_{\Sigma}^* .
\end{eqnarray}

Since 
 $$\langle A, \beta^{*}_{\Sigma} - \beta_{\Sigma}^{0}  \rangle  \, \chi_{\epsilon} ( \beta^{*}_{\Sigma} - \beta_{\Sigma}^{0}  ) = O(\epsilon),$$
 it follows by $L^2$-boundedness that the first term
 \begin{eqnarray} \label{I2.1}
 \Big\|  \int e^{i \Psi_{\Sigma}(\cdot, \beta_{\Sigma}, \beta_{\Sigma}^*, \tilde{\beta},x)/h} \, \langle A, \beta^{*}_{\Sigma} - \beta_{\Sigma}^{0}  \rangle  \, \chi_{\epsilon} ( \beta^{*}_{\Sigma} - \beta_{\Sigma}^{0}  ) \, a (\beta_{\Sigma}, \tilde{\beta},x) Tu_h(\tilde{\beta}) d\tilde{\beta}  dx d\beta_{\Sigma} d \beta_{\Sigma}^* \Big\|_{L^2(U)} \nonumber \\
 = O(\epsilon) \|T_{\Sigma} u \|_{L^2}.\end{eqnarray}

 As for the second term in (\ref{I2}), noting that $\Im \phi(\beta,x) \geq \frac{1}{C} |\beta_x - x|^2,$ it follows by a similar integration by parts argument as in the proof of Proposition \ref{WF1} (see (\ref{wf4}) ) that 
 \begin{eqnarray} \label{I2.2}
\Big\| \int e^{i \Psi_{\Sigma}(\cdot, \beta_{\Sigma}, \beta_{\Sigma}^*, \tilde{\beta},x)/h} \, \langle A,   \,  \beta^{*}_{\Sigma} - \beta_{\Sigma}^{0}  \rangle  \, \chi_{\epsilon}^+ ( \beta^{*}_{\Sigma} - \beta_{\Sigma}^{0}  ) \, a (\beta_{\Sigma}, \tilde{\beta},x) Tu_h(\tilde{\beta}) d\tilde{\beta}  dx d\beta_{\Sigma} d \beta_{\Sigma}^* \Big\|_{L^2(U)}  \nonumber \\
 = O_{\epsilon}(h^{\infty}). \hspace{1in}
 \end{eqnarray}
 Consequently, it follows from (\ref{I2.1}), (\ref{I2.2}) and (\ref{I1}) that
 \begin{equation} \label{upshot}
\|  Q(h) T_{\Sigma} u_h  - q(\alpha_{\Sigma}, \alpha_{\Sigma}^0) T_{\Sigma} u_h \|_{L^2(U)} = ( O(\epsilon) + O_{\epsilon}(h^{\infty}) ) \| T_{\Sigma} u \|_{L^2}.
\end{equation}
Since $\epsilon >0$ is arbitrary, this proves the Proposition in the first case where locally $ \Sigma \cap U = \{ (x, \xi'); \xi_n = G(x,\xi') \}.$

In the second case where $ \Sigma \cap V = \{ (x',\xi); x_n = H(x',\xi)  \},$ one argues similarly to the above but with $\beta_{\Sigma}^0 = ( \beta_{\xi'} + \beta_{\xi_n} \partial_{\beta_x'}H, \beta_{\xi_n} \partial_{\beta_{\xi}}H ).$
By constructing  a partition of unity, one can cover $\Sigma$ by finitely many sets of the form $U$ and $V$. This completes the proof.

\end{proof}

We can now turn to the proof of the 2MQER result in Theorem \ref{thm1}.

\section{2MQER: Proof of Theorem \ref{thm1}} \label{2mqer}
\begin{proof}

We first deal with the Neumann data in the LHS of (\ref{qercd}). Recall that we write $\nu=\nabla F$ for $F$ a defining function of $\Sigma$ and $X=J\nu$. By the Cauchy-Riemann equation,

\begin{align*}
    h\partial_{\nu}(e^{-\rho/h}u^{\mathbb{C}}_h)&=-(\partial_{\nu}\rho)e^{-\rho/h}u^{\mathbb{C}}_h+e^{-\rho/h}h\partial_{\nu}u^{\mathbb{C}}_h\\
    &=-(\partial_{\nu}\rho)e^{-\rho/h}u^{\mathbb{C}}_h-ie^{-\rho/h}hXu^{\mathbb{C}}_h\\
    &=-(\partial_{\nu}\rho)e^{-\rho/h}u^{\mathbb{C}}_h-ihX(e^{-\rho/h}u^{\mathbb{C}}_h) -  i(X\rho)e^{-\rho/h}u^{\mathbb{C}}_h
\end{align*}
Thus,

\begin{equation}\label{R}
-h\partial_{\nu}(Tu_h)=\Big(ihX+(\partial_{\nu}\rho+ iX\rho)\Big)Tu_h=:RTu_h.    
\end{equation}
where $R = ihX+(\partial_{\nu}\rho
+ iX\rho)$,  $X \in T\Sigma$ and so, $R$ is an h-differential operator acting  {\em tangentially} along $\Sigma.$ Similarly,

\begin{equation*}
-h\partial_{\nu}(\overline{Tu_h})=\Big(-ihX+(\partial_{\nu}\rho- iX\rho)\Big)\overline{Tu_h}=\overline{RTu_h}.    
\end{equation*}

The term involving Neumann data in (\ref{qercd}) is 

\begin{equation}\label{Neumann}
\begin{split}
&\langle a~h\partial_{\nu}e^{-\rho/h}u_h^{\mathbb{C}},h\partial_{\nu}e^{-\rho/h}u_h^{\mathbb{C}}\rangle_{L^2(\Sigma)}=\int_{\Sigma}ah\partial_{\nu}(e^{-\rho/h}u^{\mathbb{C}}_h)h\partial_{\nu}(e^{-\rho/h}\overline{u^{\mathbb{C}}_h})d\sigma\\
&=\int_{\Sigma}aRTu_h\overline{RTu_h}d\sigma
=\int_{\Sigma}R^*aRTu_h\overline{Tu_h}d\sigma
=\langle R^*aRT_{\Sigma}u_h,T_{\Sigma}u_h\rangle_{L^2(\Sigma)},
\end{split}
\end{equation}

where $d \sigma$ denotes the Riemannian hypersurface measure on $\Sigma.$

A straightforward computation together with  Proposition \ref{WF2} (in particular, the compactness of  $WF_h (T_{\Sigma} u)$) implies that

\begin{equation} \label{conjugation}
R^*aR=-ah^2X^2+2ia(\partial_{\nu}\rho)hX+a(\partial_{\nu}\rho)^2+a(X\rho)^2+O_{L^2(\Sigma)\to L^2(\Sigma)}(h).
\end{equation}

Next, we deal with the first-order term $\langle 2ah \nabla\rho(e^{-\rho/h}u_h^{\mathbb{C}}),e^{-\rho/h}u_h^{\mathbb{C}}\rangle_{L^2(\Sigma)}$ in (\ref{qercd}) using the same method. The position of $\Sigma$,  relative to the level sets of $\rho$ and the ambient complex structure is of importance here. We define the function $\theta=\theta(p)$ on $\Sigma$ to be the angle of intersection between $\Sigma$ and $\{z\in M_{\tau}^{\mathbb{C}}|\rho(z)=\rho(p)\}$ at the point $p\in \Sigma$. Similarly,  we let $\phi=\phi(p)$ be the angle between the normal vector $\nabla\rho$ and $J\nu$, where $J$ is the almost-complex structure of the Grauert tube.

\begin{remark} \label{plane}
We note that $\nu_p$ and $X = J \nu_p$ span a real 2-plane in $T_p M_\tau^\C$ at any point $p \in \Sigma$. However, $(\nabla \rho)_p$ does not necessarily lie in span$(\nu_p , X_p)$ and so, $\phi$ and $\theta$ defined above are independent variables. However, the range of $(\theta,\phi)\in [0,\pi]^2$ is contained in a diamond-shaped region (see Figure 1 and thereafter).
\end{remark} 

We thus have, by definition

\begin{equation}\label{theta,phi}
\langle\nabla\rho,\nu\rangle=|\nabla\rho|\cos\theta,\quad \langle \nabla\rho,J\nu\rangle=|\nabla \rho|\cos\phi.
\end{equation}

Decompose the vector field $(\nabla\rho)|_{\Sigma}$ into two parts: $(\nabla\rho)^{T}$ tangential to $\Sigma$ and $(\nabla\rho)^{\nu}$ normal to $\Sigma$. The tangential component of $\nabla \rho$ to $\Sigma$ at $p \in \Sigma$ then has length

\begin{equation} \label{tangential}
|(\nabla \rho)^{T} |= |\nabla \rho| \, \sin \theta.
\end{equation}

\begin{figure}[ht!]
\centering
\includegraphics[width=0.6\textwidth]{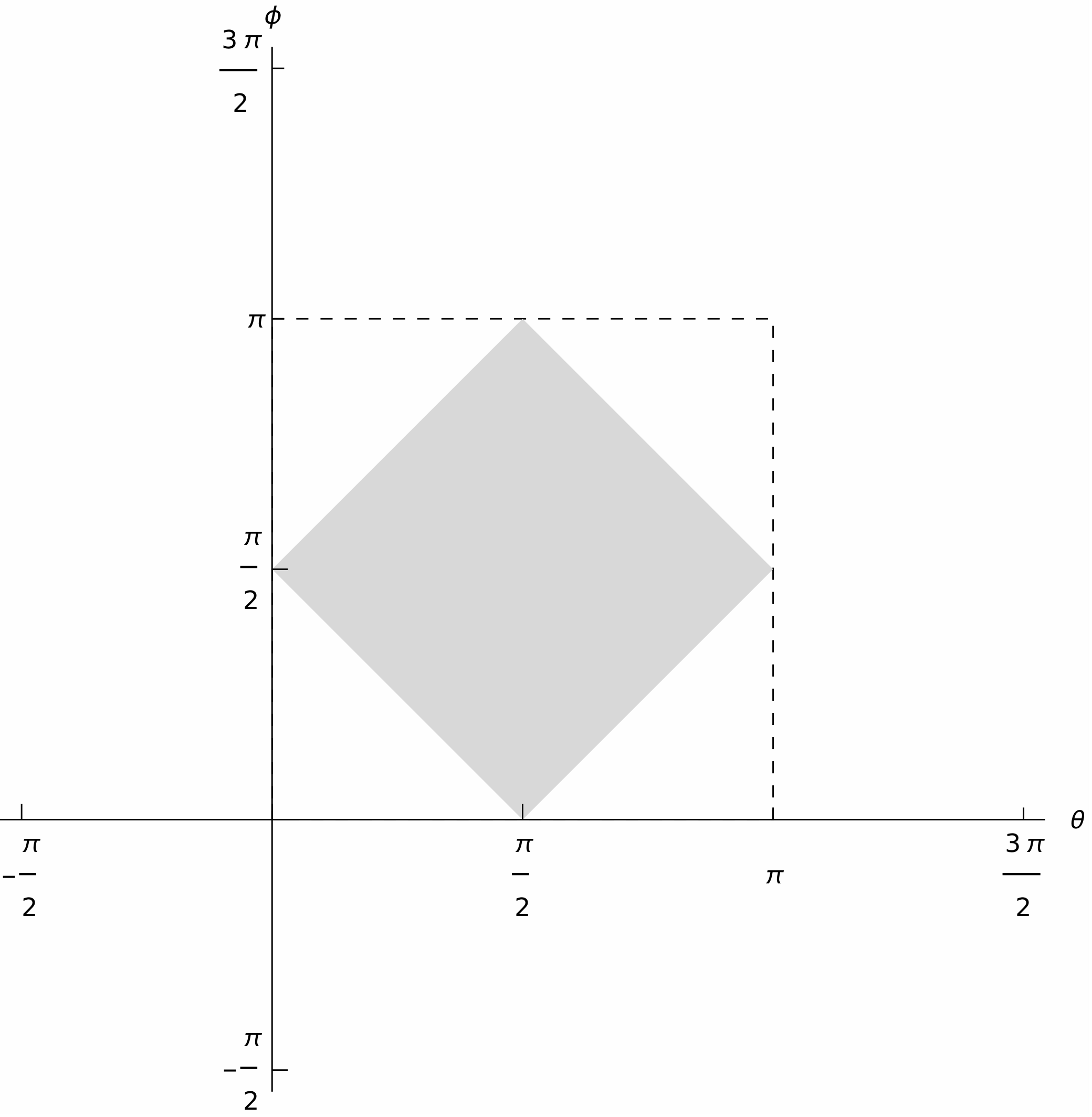}
\caption{Gray: Admissible region of $(\theta,\phi)$.}
\label{Fig1}
\end{figure}

\begin{figure}[ht]
\centering{
\resizebox{140mm}{!}
{
\begingroup%
  \makeatletter%
  \providecommand\color[2][]{%
    \errmessage{(Inkscape) Color is used for the text in Inkscape, but the package 'color.sty' is not loaded}%
    \renewcommand\color[2][]{}%
  }%
  \providecommand\transparent[1]{%
    \errmessage{(Inkscape) Transparency is used (non-zero) for the text in Inkscape, but the package 'transparent.sty' is not loaded}%
    \renewcommand\transparent[1]{}%
  }%
  \providecommand\rotatebox[2]{#2}%
  \newcommand*\fsize{\dimexpr\f@size pt\relax}%
  \newcommand*\lineheight[1]{\fontsize{\fsize}{#1\fsize}\selectfont}%
  \ifx\svgwidth\undefined%
    \setlength{\unitlength}{260.27971643bp}%
    \ifx\svgscale\undefined%
      \relax%
    \else%
      \setlength{\unitlength}{\unitlength * \real{\svgscale}}%
    \fi%
  \else%
    \setlength{\unitlength}{\svgwidth}%
  \fi%
  \global\let\svgwidth\undefined%
  \global\let\svgscale\undefined%
  \makeatother%
  \begin{picture}(1,0.48308455)%
    \lineheight{1}%
    \setlength\tabcolsep{0pt}%
    \put(0,0){\includegraphics[width=\unitlength,page=1]{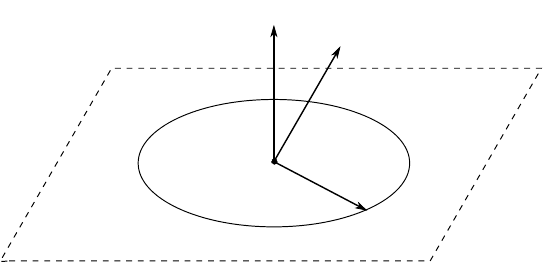}}%
    \put(0.45896554,0.45351906){\makebox(0,0)[lt]{\lineheight{1.25}\smash{\begin{tabular}[t]{l}$\nu$\end{tabular}}}}%
    \put(0.60861271,0.42414909){\makebox(0,0)[lt]{\lineheight{1.25}\smash{\begin{tabular}[t]{l}$\nabla\rho$\end{tabular}}}}%
    \put(0.67014979,0.06611484){\makebox(0,0)[lt]{\lineheight{1.25}\smash{\begin{tabular}[t]{l}$J\nu$\end{tabular}}}}%
    \put(0,0){\includegraphics[width=\unitlength,page=2]{drawing.pdf}}%
    \put(0.51211137,0.25002694){\makebox(0,0)[lt]{\lineheight{1.25}\smash{\begin{tabular}[t]{l}$\theta$\end{tabular}}}}%
    \put(0.53518753,0.20037766){\makebox(0,0)[lt]{\lineheight{1.25}\smash{\begin{tabular}[t]{l}$\phi$\end{tabular}}}}%
    \put(0.01841561,0.28708909){\makebox(0,0)[lt]{\lineheight{1.25}\smash{\begin{tabular}[t]{l}$T_p\Sigma$\end{tabular}}}}%
    \put(0.14288844,0.07450626){\makebox(0,0)[lt]{\lineheight{1.25}\smash{\begin{tabular}[t]{l}$S^{2n-2}$\end{tabular}}}}%
  \end{picture}%
\endgroup%
{
\begingroup%
  \makeatletter%
  \providecommand\color[2][]{%
    \errmessage{(Inkscape) Color is used for the text in Inkscape, but the package 'color.sty' is not loaded}%
    \renewcommand\color[2][]{}%
  }%
  \providecommand\transparent[1]{%
    \errmessage{(Inkscape) Transparency is used (non-zero) for the text in Inkscape, but the package 'transparent.sty' is not loaded}%
    \renewcommand\transparent[1]{}%
  }%
  \providecommand\rotatebox[2]{#2}%
  \newcommand*\fsize{\dimexpr\f@size pt\relax}%
  \newcommand*\lineheight[1]{\fontsize{\fsize}{#1\fsize}\selectfont}%
  \ifx\svgwidth\undefined%
    \setlength{\unitlength}{260.27971643bp}%
    \ifx\svgscale\undefined%
      \relax%
    \else%
      \setlength{\unitlength}{\unitlength * \real{\svgscale}}%
    \fi%
  \else%
    \setlength{\unitlength}{\svgwidth}%
  \fi%
  \global\let\svgwidth\undefined%
  \global\let\svgscale\undefined%
  \makeatother%
  \begin{picture}(1,0.59157057)%
    \lineheight{1}%
    \setlength\tabcolsep{0pt}%
    \put(0,0){\includegraphics[width=\unitlength,page=1]{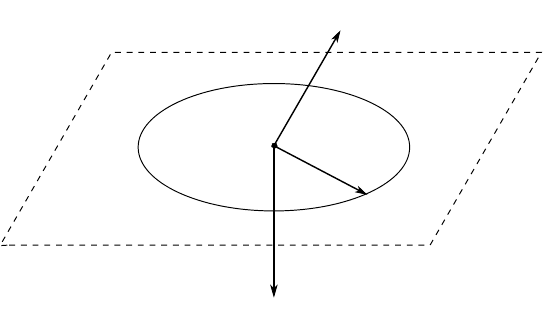}}%
    \put(0.4561684,0.00677229){\makebox(0,0)[lt]{\lineheight{1.25}\smash{\begin{tabular}[t]{l}$\nu$\end{tabular}}}}%
    \put(0.60861271,0.56200509){\makebox(0,0)[lt]{\lineheight{1.25}\smash{\begin{tabular}[t]{l}$\nabla\rho$\end{tabular}}}}%
    \put(0.67014979,0.20397085){\makebox(0,0)[lt]{\lineheight{1.25}\smash{\begin{tabular}[t]{l}$J\nu$\end{tabular}}}}%
    \put(0,0){\includegraphics[width=\unitlength,page=2]{drawing2.pdf}}%
    \put(0.53099182,0.33263938){\makebox(0,0)[lt]{\lineheight{1.25}\smash{\begin{tabular}[t]{l}$\phi$\end{tabular}}}}%
    \put(0.01841561,0.42494509){\makebox(0,0)[lt]{\lineheight{1.25}\smash{\begin{tabular}[t]{l}$T_p\Sigma$\end{tabular}}}}%
    \put(0.14288844,0.21236227){\makebox(0,0)[lt]{\lineheight{1.25}\smash{\begin{tabular}[t]{l}$S^{2n-2}$\end{tabular}}}}%
    \put(0,0){\includegraphics[width=\unitlength,page=3]{drawing2.pdf}}%
    \put(0.52469842,0.22215234){\makebox(0,0)[lt]{\lineheight{1.25}\smash{\begin{tabular}[t]{l}$\theta$\end{tabular}}}}%
  \end{picture}%
\endgroup%
}}
\caption{$\phi$ attains its max and min when $J\nu\in \text{span} (\nu,\nabla\rho)$.}
\label{Fig2}
}
\end{figure}

We compute
\begin{align}\label{1st}
    &\langle 2ah\nabla\rho(e^{-\rho/h}u_h^{\mathbb{C}}),e^{-\rho/h}u_h^{\mathbb{C}}\rangle_{L^2(\Sigma)}\nonumber\\
   =&\langle 2ah\nabla\rho(e^{-\rho/h})u_h^{\mathbb{C}},e^{-\rho/h}u_h^{\mathbb{C}} \rangle+\langle 2ahe^{-\rho/h}(\nabla \rho)u_h^{\mathbb{C}},e^{-\rho/h}u_h^{\mathbb{C}}\rangle\nonumber\\
   =&\langle -2a|\nabla\rho|^2e^{-\rho/h}u_h^{\mathbb{C}},e^{-\rho/h}u_h^{\mathbb{C}}\rangle+\langle 2ahe^{-\rho/h}[(\nabla\rho)^{T}+(\nabla\rho)^{\nu}]u_h^{\mathbb{C}},e^{-\rho/h}u_h^{\mathbb{C}}\rangle\\
   =&\langle -2a|\nabla\rho|^2e^{-\rho/h}u_h^{\mathbb{C}},e^{-\rho/h}u_h^{\mathbb{C}}\rangle+\langle 2ahe^{-\rho/h}\underbrace{(-\langle\nabla\rho,\nu\rangle iX+(\nabla\rho)^{T})}_{=:Y}u_h^{\mathbb{C}},e^{-\rho/h}u_h^{\mathbb{C}}\rangle\nonumber\\
   =&\big\langle a\Big(-2|\nabla\rho|^2+2hY-2i|\nabla\rho|^2\cos\theta\cos\phi+2|\nabla\rho|^2 \sin^{2}\theta\Big)e^{-\rho/h}u_h^{\mathbb{C}},e^{-\rho/h}u_h^{\mathbb{C}}\big\rangle_{L^2(\Sigma)}\nonumber
\end{align}
where we used Cauchy-Riemann equation (\ref{CReqs}) in the derivation of $Y$. The last equality follows from the fact that

\begin{align}\label{Ye-eY}
    &[Y,  e^{-\rho/h}] = Y(e^{-\rho/h})\nonumber\\
    =&-\langle\nabla\rho,\nu\rangle iXe^{-\rho/h}+(\nabla\rho)^{T}e^{-\rho/h}\nonumber\\
    =&\frac{1}{h}\langle\nabla\rho,\nu\rangle ie^{-\rho/h}\underbrace{(X\rho)}_{=\langle \nabla\rho,J\nu\rangle}-\frac{1}{h} e^{-\rho/h}\underbrace{((\nabla\rho)^{T}\rho)}_{=\langle (\nabla\rho)^{T},\nabla\rho\rangle}\\
    =&\frac{1}{h}ie^{-\rho/h}|\nabla\rho|\cos\theta |\nabla\rho|\cos\phi-\frac{1}{h}e^{-\rho/h}|\nabla\rho|^2\sin^2\theta.\nonumber
\end{align}
 In the last line of (\ref{Ye-eY}) we have used that
$\langle (\nabla \rho)^{T}, \nabla \rho \rangle = | (\nabla \rho)^T| |\nabla \rho| \cos (\frac{\pi}{2} - \theta) =  | (\nabla \rho)^T| |\nabla \rho| \sin \theta = |\nabla \rho|^2 \sin^2 \theta.$ 
From  (\ref{conjugation}) it also follows that 

\begin{equation}\label{2nd}
R^*aR=-ah^2X^2+2ia|\nabla\rho|\cos\theta \, hX+a|\nabla\rho|^2\cos^2\theta+a|\nabla\rho|^2\cos^2\phi + O_{L^2(\Sigma) \to L^2(\Sigma)}(h).  \end{equation}
To get the left-hand side of (\ref{qercd}), we sum up all the terms in (\ref{1st}), (\ref{2nd}),  along with  $ \langle a|\nabla\rho|^2 e^{-\rho/h} u_h^{\C}, e^{-\rho/h} u_h^{\C} \rangle$ and $\langle a h^2\Delta_{\Sigma}(e^{-\rho/h}u_h^{\C}),e^{-\rho/h}u_h^{\C}\rangle$ to get that

\begin{eqnarray}\label{explicit}
  Q_a(h) = a\Big(h^2\Delta_{\Sigma}-h^2X^2+2h(\nabla\rho)^{T}+|\nabla\rho|^2(\cos^2\phi+\sin^2\theta-2i\cos\theta\cos\phi)\Big) \\ \nonumber 
    +  O_{L^2(\Sigma) \to L^2(\Sigma)}(h). 
\end{eqnarray}

We have reduced the left-hand side of (\ref{qercd}) into a tangential $h\Psi$DO acting on $T_{\Sigma}u_h$. Finally, in view of the wavefront localization in Proposition  \ref{WF2},

$$
\langle Q_a(h) T_\Sigma u, T_\Sigma u \rangle_{L^2(\Sigma)} = \langle \chi_\Sigma(h)^* Q_a(h) \chi_{\Sigma}(h) T_\Sigma u, T_\Sigma u \rangle_{L^2(\Sigma)} + O(h^\infty)
$$
where $\chi_{\Sigma} \in C^{\infty}_0$ equals $1$ near the compact set ${\mathcal W}_{\Sigma}$ in Proposition \ref{WF2}. The theorem then follows with

\begin{equation} \label{test op}
{\mathcal P}_{\Sigma, a}(h) =  \chi_\Sigma(h)^* Q_a(h) \chi_{\Sigma}(h) \in \Psi_h^0(\Sigma).
\end{equation}

\end{proof}

 \begin{remark} \label{multiplierQER}
Given any $a \in C^{\infty}(\Sigma),$ it follows from Proposition \ref{multiplier} that 
 \begin{equation} \label{mult}
 \langle a {\mathcal P}_{\Sigma,1}(h) T_{\Sigma} u, T_{\Sigma} u \rangle_{L^2(\Sigma)} =  \langle a f_{\Sigma} T_{\Sigma} u, T_{\Sigma} u \rangle_{L^2(\Sigma)}  + o(1) \| T_{\Sigma} u \|^2_{L^2(\Sigma)},
 \end{equation}
 with $f_{\Sigma} = \sigma({\mathcal P}_{\Sigma,1}) |_{{\mathcal W}_{\Sigma}'}$ with ${\mathcal W}_{\Sigma}'$ in (\ref{refinement}).
 
  Then since ${\mathcal P}_{\Sigma,a} = a {\mathcal P}_{\Sigma,1},$ it follows from Theorem \ref{thm1} and (\ref{mult})  that under the QE  assumption on the sequence $(u_h),$ and  for any $a \in C^{\infty}(\Sigma),$ 
  \begin{equation} \label{multiplierthm}
\langle a f_{\Sigma} T_{\Sigma} u_h, T_{\Sigma}u_h \rangle_{L^2(\Sigma)}  + o(1) \| T_{\Sigma} u_h \|_{L^2(\Sigma)}^2 = \int_{S^*M \cap \Sigma}  a  \, q  d\mu_{\Sigma} + o(1).
\end{equation}
Here, (\ref{multiplierthm}) can be viewed as  the  Toeplitz multiplier version of QER for the restrictions $T_{\Sigma} u_h.$ 
\end{remark}

\section{ $L^2$-restriction bounds for $\| T_{\Sigma} u_h \|_{L^2(\Sigma)}$: proof of Theorem \ref{bounds}}\label{UL}

We first prove the crude upper bound (\ref{crude upper bound}) of order $h^{-1/2}$. The proof is essentially a Sobolev restriction argument.\

\begin{lemma}\label{lemma of crude}
    Under the assumptions of Theorem 2, we have $\|T_{\Sigma}u_h\|_{L^2(\Sigma)}\leq C_{\Sigma}h^{-1/2}$. 
\end{lemma}

\begin{proof}
    
To prove the upper bounds, we note that by Sobolev restriction, if $\Sigma = \partial \Omega$ with $\Omega \subset B_\tau^*,$ 

\begin{equation} \label{upper1}
\| T_{\Sigma} u \|_{L^2(\Sigma)} \leq C_1 \| T u \|_{H^{\frac{1}{2}}(\Omega)}.
\end{equation}

Since 

$$ \| T u \|_{H^{\frac{1}{2}}(\Omega)}^2 \leq C  \big| \langle Op_1 ( \langle (x^*,\xi^*) \rangle) T u, Tu \rangle_{L^2(\Omega)} \big|$$
after rescaling $(x^*,\xi^*) \to h^{-1} (x^*,\xi^*)$ in the fiber variables, it follows that

\begin{align} \label{upper2}
 \| T u \|_{H^{\frac{1}{2}}(\Omega)}^2 
 &\leq C h^{-1} \langle Op_h ( ( |x^*|^2  +  |\xi^*|^2 + h )^{1/2} ) T u, Tu \rangle_{L^2(\Omega)} \nonumber \\
  &= C h^{-1} \langle \chi_{{\mathcal W}}^*  ( ( |x^*|^2 + |\xi^*|^2 + h )^{1/2} ) \chi_{{\mathcal W}} T u, Tu \rangle_{L^2(\Omega)} + O(h^{\infty}),
 \end{align}
where ${\mathcal W} \subset T^*B_{\tau}^*M$ is the compact set in Proposition \ref{WF} and $\chi_{{\mathcal W}} \in C^{\infty}_0(T^*B_{\tau}^*)$ equals 1 in a neighbourhood of ${\mathcal W}.$
    
Since $\chi_{{\mathcal W}}^*  ( ( |x^*|^2 + |\xi^*|^2 + h )^{1/2} ) \chi_{{\mathcal W}} \in \Psi_h^0(B_\tau^*M)$, it follows by $L^2$-boundedness in (\ref{upper2}) that
    
$$
\| T u \|_{H^{\frac{1}{2}}(\Omega)}^2 \leq C' h^{-1}
$$
and  in view of (\ref{upper1}) we are done.

\end{proof}

Next, we make some preparations for the uniform upper bound (\ref{comparable}). We will need the following elementary result:

\begin{lemma}\label{sum of squares}
Let $(M,g)$ be a Riemannian manifold, and let $X$ be a vector field with $|X|_g=1$. Then $\sigma(h^2\Delta-h^2X^2)\leq 0$.
\end{lemma}
\begin{proof}
    In terms of local coordinates, we let  $X=\sum_{j=1}^nc_j(x)\partial_{x_j}$. Consequently, $-h^2X^2\\=-\sum_{j,k}c_jc_kh^2\partial_{x_j}\partial_{x_k}+O_{L^2\to L^2}(h) =-(\sum_jc_jh\partial_{x_j})^2+ O_{L^2\to L^2}(h)$. As a result, 
    
    $$
    \sigma(-h^2X^2)=(\sum_{j=1}^nc_j\xi_j)^2=(\alpha(X))^2,
    $$
    where $\alpha=\xi dx$ is the canonical one-form in $T^*M$.   Then, the inequality in the statement of the Lemma is simply that $\alpha(X)^2\leq |\xi|^2_g$. That is exactly the Cauchy-Schwarz inequality for pairing between a 1-form and a vector field, namely

    $$
    |\alpha(X)|=|\langle \alpha^{\#},X\rangle_g|\leq |\alpha^{\#}|_g|X|_g=|\xi|_g,
    $$
    since $|X|_g=1$, and $|\alpha^\#|_g=|\xi|_g$.
    
\end{proof}

At $z\in \Sigma$, consider the complex tangential subspace to $\Sigma$, given by
\begin{equation}
    \mathcal{H}_z=\{\nu_z,X_z\}^\perp \subset T_{z} \Sigma,\quad X=J\nu.
\end{equation}
It is $J$-invariant. For any real vector field $V\in \mathcal{H}_z$, we have the following weighted Cauchy-Riemann equation:
\begin{equation}\label{WF L}
    (hJV-ihV+JV\rho-iV\rho)T_\Sigma u_h=0.
\end{equation}
Set 
$$ {\mathcal L}:= hJV-ihV+JV\rho-iV\rho.$$

 Since  $\mathcal{L}(T_\Sigma u_h)=0$,  it follows by a standard $h$-elliptic parametrix construction (\cite{Zw12} section 4) that
\begin{align}\label{WF new}
    \WF_h(T_\Sigma u_h)&\subset\{\Re\sigma({\mathcal L})= \Im\sigma({\mathcal L})=0\}\nonumber\\
    &=\{\sigma(-ihV)=-JV\rho,\,\sigma(-ihJV)=V\rho\}.
\end{align}

Using (\ref{WF L}), one can write the  restriction of principal symbol of $\sigma(h^2\Delta-h^2X^2)$ on $WF_h(T_{\Sigma}u_h)$ as a simple geometric function defined in terms of $(\theta,\phi).$

\begin{lemma}\label{lem sum of squares}
    For any $ (\alpha_{\Sigma}, \alpha_{\Sigma}^*) \in \WF_h(T_\Sigma u_h)$, 
    $$\sigma(h^2\Delta_{\Sigma}-h^2X^2)(\alpha_{\Sigma}, \alpha_{\Sigma}^*) = - |\nabla\rho|^2(\alpha_{\Sigma}) \cdot \big(1-\cos^2\theta(\alpha_{\Sigma})  - \cos^2\phi (\alpha_{\Sigma}) \big).$$ 
\end{lemma}
\begin{proof}
    We choose the (real) orthonormal frame in $\mathcal{H}$:
    $$
    V_1,\,JV_1,\,\cdots,\, V_{n-1},\, JV_{n-1}.
    $$
    Adding $X=J\nu$, one obtains an orthonormal frame in $T\Sigma$; Further adding $\nu$, one obtains an orthonormal frame of $TM_\tau^\C$. 

     It is well known \cite{Lee18,Mor07,Ura12} that in an orthonormal frame, the Laplacian can be written as a sum of squares to leading order; indeed (recall the sign convention $\Delta$ is non-positive),
    \begin{equation}
        h^2\Delta=\sum_j\big( h^2X_j^2+h(\text{div} X_j)hX_j\big).
    \end{equation}
    In the case of $h^2\Delta_\Sigma$ and the frame $\{V_1,JV_1, V_2, JV_2,...,V_{n-1}, J V_{n-1},X\}$,

    \begin{equation}
        \sigma(h^2\Delta_\Sigma-h^2X^2)=-\sum_{j=1}^{n-1}(\sigma(-ihV_j)^2+\sigma(-ihJV_j)^2),
    \end{equation}
and consequently, from (\ref{WF new}), it follows that  for $(\alpha_\Sigma,\alpha_{\Sigma}^*) \in WF_h(T_\Sigma u_h),$

\begin{align}
    -\sigma(h^2\Delta_\Sigma-h^2X^2)(\alpha_\Sigma, \alpha_\Sigma^*)
    &=|\nabla\rho|^2(\alpha_\Sigma)-(\p_\nu\rho)^2(\alpha_\Sigma)-(X\rho)^2(\alpha_\Sigma) \nonumber\\
    &=|\nabla\rho|^2(\alpha_\Sigma)(1-\cos^2\theta(\alpha_\Sigma)-\cos^2\phi(\alpha_\Sigma)). 
\end{align}


\end{proof}

Substitution of the formula in Lemma \ref{lem sum of squares} into (\ref{explicit}) gives the following key formula for $\Re \sigma(Q_{1}) |_{ WF_{h}(T_{\Sigma} u_h)}:$

\begin{eqnarray} \label{key formula}
\Re \sigma(Q_{1})(\alpha_\Sigma,\alpha_\Sigma^*) &= |\nabla \rho|^2 \big( \cos^2 \phi + \sin^2\theta - (1- \cos^2 \theta - \cos^2 \phi) \big)(\alpha_\Sigma), \nonumber \\
&= 2 |\nabla \rho|^2(\alpha_\Sigma) \cdot \cos^2\phi(\alpha_\Sigma), \quad  (\alpha_\Sigma, \alpha_\Sigma^*)  \in WF_h(T_\Sigma u_h).
\end{eqnarray}

From (\ref{key formula}), the $h$-ellipticity condition is then $\phi\neq \pi/2$, or equivalently,
\begin{equation}\label{nzero}
  \quad X \rho = \langle \nabla\rho,J\nu\rangle \neq 0.
\end{equation}

To summarize, it follows from (\ref{key formula}) that

\begin{equation} \label{ellipticity}
\Re \sigma(Q_1)(\alpha_\Sigma,\alpha_\Sigma^*) >0, \quad (\alpha_\Sigma,\alpha_\Sigma^*) \in WF_h(T_\Sigma u_h) \setminus {\mathcal C},
\end{equation}
where ${\mathcal C} = \{ X\rho(\alpha_\Sigma) = 0 \} \subset \Sigma.$

\pagebreak


On the other hand, we know apriori that 

\begin{align}\label{diamond}
\phi\in[\frac{\pi}{2}-\theta,\frac{\pi}{2}+\theta],\quad&\text{for }\theta\in [0,\frac{\pi}{2}],\nonumber\\
\phi\in[-\frac{\pi}{2}+\theta,\frac{3\pi}{2}-\theta],\quad&\text{for }\theta\in [\frac{\pi}{2},\pi],
\end{align}
which gives the diamond shape in Figure 1. We refer to this as the {\em admissible} region.

Figure 2 shows that if we fix $\nu_p$ and $(\nabla\rho)_p$, then $J\nu_p$ lies in the unit sphere within $T_p\Sigma$. It follows from simple geometry that (\ref{diamond}) holds.\\

We now complete the proof of Theorem \ref{bounds}.

\begin{proof}[Proof of Theorem \ref{bounds}]

 {\em Lower bounds:}   By Proposition \ref{WF2}, $WF_h(Tu_h|_{\Sigma}) \subset {\mathcal W}_{\Sigma}$ where the latter set is compact and by $C^{\infty}$ Urysohn,  we can choose $\chi_{\Sigma} \in C_0^{\infty}(T^*\Sigma)$ with $\chi_{\Sigma}(x,\xi,x^*,\xi^*)=1$  near ${\mathcal W}_{\Sigma}$ and let $\tilde{\chi}_{\Sigma} \in C^{\infty}_0(T^*\Sigma)$ with $\tilde{\chi}_{\Sigma} \Supset \chi_{\Sigma}.$ ( The notation $a\Supset b$ means $\{a=1\}\supset \supp b$.) We denote the corresponding quantizations by $\chi(h) \in \Psi_h^{0}(\Sigma)$ and $\tilde{\chi}(h) \in \Psi_h^0(\Sigma).$

    Then setting the test symbol $a =1$ in Theorem \ref{thm1}, it follows that with ${\mathcal P}_{\Sigma, a}(h) \in \Psi_h^0(\Sigma)$ in (\ref{test op}),
    
 \begin{eqnarray} \label{lower1}
\big| \int_{S^*M \cap \Sigma}q~ d\mu_{\Sigma} \big| \leq  \big|  \langle {\mathcal P}_{\Sigma,1}(h) T_{\Sigma} u, T_{\Sigma} u \rangle_{L^2(\Sigma)}  \big| + o(1) \nonumber \\
\leq c'_{\Sigma}   \| T_{\Sigma} u  \|^2_{L^2(\Sigma)} + o(1) 
\end{eqnarray}
by Cauchy-Schwarz and $L^2$-boundedness of ${\mathcal P}_{\Sigma,1}(h).$ 

On the LHS in (\ref{lower1}), $ \big| \int_{S^*M \cap \Sigma} q~d\mu_{\Sigma} \big|  \geq c''_{\Sigma} >0$ since $\Sigma$ is assumed to be in general position and so, the lower bound
$$ \| T_{\Sigma} u  \|_{L^2(\Sigma)}  \geq c_{\Sigma} >0$$
follows from (\ref{lower1})  for $h\in (0,h_0]$ sufficiently small since the $o(1)$-error can then be absorbed in the LHS.

{\em Upper bounds.}
The crude, universal $O(h^{-1/2})$-upper bound has already been established in Lemma \ref{lemma of crude}. In the following, we continue to write ${\mathcal C} = \{ X \rho =0 \} \subset \Sigma.$ We now prove the  locally uniform upper bound on $U\Subset\Sigma\setminus {\mathcal C}$ using the $h$-microlocal ellipticity result in (\ref{ellipticity}). To begin, let $\tilde{U} \subset \Sigma \setminus {\mathcal C}$ be open with $\tilde{U} \Supset U$ and cutoffs $\psi \in C^{\infty}_0(\tilde{U})$ and with $\tilde{\psi} |_{U} =1$ 

Since $WF_h(T_\Sigma u)$ is compact  and $\sigma(Q_1) \in C^{\infty}(T^*\Sigma)$, it follows from (\ref{ellipticity}) that there exists an open neighbourhood $\tilde{\mathcal O} \supset WF_h(T_\Sigma u)$ such that
\begin{equation} \label{upper1'}
\Re \sigma(Q_1)(\alpha_\Sigma, \alpha_\Sigma^*) \geq c_1 >0, \quad (\alpha_\Sigma,\alpha_\Sigma^*) \in \tilde{\mathcal O}, \,\, \alpha_\Sigma \in \text{supp} \, \psi.
\end{equation}

Now let $ {\mathcal O} \Subset \tilde{{\mathcal O}} $ be a slightly smaller open neighbourhood and by $C^{\infty}$ Urysohn, we introduce the cutoff  $\chi_\Sigma \in C^{\infty}_0(T^*\Sigma)$ with  $\chi_\Sigma  \,|_{{\mathcal O}} =1$ and supp $\chi_{\Sigma} \subset \tilde{\mathcal O}.$

Set 
$$ v_h := \psi \, \chi_{\Sigma}(h) T_\Sigma u_h.$$

Then with $\Re Q_1 = \frac{1}{2} (Q_1 + Q_1^*)$ and using (\ref{upper1'}), it follows by sharp G\aa rding (\cite{Zw12} Theorem 9.11) that,

\begin{equation} \label{upper2'}
\langle \Re Q_1  v, v \rangle_{L^2(\Sigma)} \geq c_1 \| v \|^2_{L^2(\Sigma)} - c_2 h \| v\|^2_{L^2(\Sigma)} + O(h^{\infty}),
\end{equation}
since $\| (I-\chi_{\Sigma}(h)) T_\Sigma u_h \|_{L^2} = O(h^{\infty}).$ 

Next, we write 
\begin{equation} \label{commutator1}
 \langle Q_1  v, v \rangle_{L^2(\Sigma)} = \langle \psi^2 Q_1  T_\Sigma u_h, T_\Sigma u_h  \rangle_{L^2(\Sigma)} - \langle \psi [\psi, Q_1]  T_\Sigma u_h, T_\Sigma u_h  \rangle_{L^2(\Sigma)}+O(h^\infty). 
\end{equation}
Since $ \psi [\psi, Q_1]  \in \Psi^{-1}_h(\Sigma),$ by $L^2$-boundedness and Lemma \ref{lemma of crude},

$$| \langle \psi [\psi, Q_1]  T_\Sigma u_h, T_\Sigma u_h  \rangle_{L^2(\Sigma)} | \leq C h \| T_{\Sigma} u_h \|_{L^2(\Sigma)}^2 \leq c_3. $$
In view of (\ref{commutator1}) and (\ref{upper2'}) it follows that

\begin{equation} \label{upper3}
\Re\langle \psi^2 Q_1  T_\Sigma u_h,  T_\Sigma u_h \rangle_{L^2(\Sigma)} \geq c_1 \| v \|^2_{L^2(\Sigma)} - c_2 h \| v\|^2_{L^2(\Sigma)} - c_3.
\end{equation}

On the other hand, by Theorem \ref{thm1},

\begin{equation} \label{upper4}
\langle \psi^2 Q_1  T_\Sigma u_h,  T_\Sigma u_h \rangle_{L^2(\Sigma)} = \int_{S^*M \cap \Sigma} \psi^2 q d\mu_{\Sigma} + o(1),
\end{equation} 
and so, taking real parts,

\begin{equation} \label{upper5}
\Re\langle \psi^2   Q_1  T_\Sigma u_h,  T_\Sigma u_h \rangle_{L^2(\Sigma)}  \leq c_4.
\end{equation}
Substituting  (\ref{upper5})  in (\ref{upper3}) gives

\begin{equation} \label{upper6}
c_4 \geq  c_1 \| v \|^2_{L^2(\Sigma)} - c_2 h \| v\|^2_{L^2(\Sigma)}  - c_3.
\end{equation}
It follows from (\ref{upper6}) that for $h>0$ sufficiently small,

\begin{equation} \label{UPPER}
\| \psi T_\Sigma u_h \|^2_{L^2(\Sigma)} = \| v \|^2_{L^2(\Sigma)}+O(h^\infty) = O(1),
\end{equation} 
and so,

$$
\| T_{\Sigma} u_h \|^2_{L^2(U)}  \leq \| \psi T_{\Sigma} u_h \|^2_{L^2}  = O(1).
$$

\end{proof}

\appendix \label{q}
\section{The explicit expression for $q$}
The symbol $q$ in the complex QER theorem is provided in formulas (5.34) and (6.10) in \cite{CT24}. For completeness, we also list it here.  Lemma \ref{general} is new. Recall the QER of Cauchy data in the complex setting (\ref{qercd}):
\begin{equation}
    \begin{split}
        \langle a(h^2\Delta_{\Sigma}+2h\nabla\rho+|\nabla\rho|^2+h\Delta\rho)e^{-\rho/h}u^{\mathbb{C}}_h,e^{-\rho/h}u^{\mathbb{C}}_h\rangle&_{L^2(\Sigma)}\\
        +\langle ah\partial_\nu (e^{-\rho/h}u^{\mathbb{C}}_h),h\partial_{\nu}(e^{-\rho/h}u^{\mathbb{C}}_h)\rangle&_{L^2(\Sigma)}\\
        \sim_{h\to 0^+}e^{1/h}&\int_{\Sigma\cap S^*M}a  \,q \,d\mu_{\Sigma}.
    \end{split}
\end{equation}
where $d\mu_{\Sigma}$ is the measure on $\Sigma\cap S^*M$ induced from the Liouville measure on $S^*M$, the density $q=q_1+q_2$ is
\begin{equation}\label{q1}
    \begin{split}
    q_1(x,\xi)=&8|b_0(x,-2\xi,x)|^2(\partial_{\beta}\varphi)(x,-2\xi,x)(\xi\cdot\partial_x\beta(x,-2\xi))
    \end{split}
\end{equation}

and
\begin{equation}\label{q2}
q_2(x,\xi)=|b_0(x,\xi,x)|^2\big((\xi\cdot\partial_{\beta}x)^2-\rho_{\beta}(\xi\cdot\partial_{\beta}x)\big).
\end{equation}

In these expressions, $b_0$ is the leading term in the amplitude of the FBI transform $T=T_{hol}$,  $\varphi$ is the phase of $T$ and $\beta$ is the normalized defining function for $\Sigma.$

We now prove the following 

\begin{lemma} \label{general}
Let $H \subset M$ be a fixed hypersurface of the base manifold, $M,$ and  $\Sigma \subset B^*M$ be a hypersurface in the Grauert tube with the property that
$$ \tilde{g}( \Sigma \cap S^*M, S_H^*M) < \epsilon_0,$$
where $\tilde{g}$ is the K\"ahler-Riemannian metric on $B^*M.$
Then, provided $\epsilon_0 >0$ is sufficiently small, $\Sigma$ is a hypersurface in general position; that is,
$$ \int_{\Sigma \cap S^*M} q \, d \mu_{\Sigma} \neq 0.$$
\end{lemma}

\begin{proof} Consider the special case where $\Sigma = B_H^*M$ so that $\Sigma \cap S^*M = S_H^*M.$ Fix $\delta >0$ small, and pick points $p_j \in \Sigma, j=1,...,N$ with $\Sigma = B_H^*M = \cup_{j=1}^N B_{\delta}(p_j),$ where $B_{\delta}(p_j)$ is a ball centered at $p_j$ and radius $\delta >0.$ Let $\chi_j \in C^{\infty}(\Sigma); j=1,...,N$ be a partition of unity subordinate to this covering. 
We work locally in a component ball $B_{\delta}(p).$ Given the canonical projection $\pi: B^*M \to M$ we let $x: \pi(B_\delta(p)) \to \R^n$ be geodesic normal coordinates centered at $\pi(p) \in H$ so that $x(\pi(p)) = 0$ and by possibly making a linear change of $x$-coordinates fixing $\pi(p) \in H,$ we can assume that 
$\nu_H(\pi(p)) =  \partial_{x_n}.$ Let $\xi \in T^*_xM$ denote the corresponding fiber coordinates.  The defining function in this case is then of the form
$$\beta_0(x,\xi) = \beta_0(x) = x_n + O(|x|^2); \quad  (x,\xi) \in B_{\delta},$$

and so the normal vector field to $B_H^*M = \{ \beta_0 (x) = 0 \}$ is of the form
\begin{equation} \label{vfield}
\partial_{\beta_0} = \partial_{x_n} + A(x) \cdot \partial_{x}; \quad |A(x)| = O(x).
\end{equation}

Since $|b_0(x,-2\xi,x)|  = 1,$ and
$$ \partial_{\beta_0}\varphi(x,-2\xi,x) = -2 \xi_n - 2 A(x) \cdot \xi = -2\xi_n + O(\delta),$$
 it follows from (\ref{q1}) and (\ref{vfield}) that for all $(x,\xi) \in B_\delta(p) \cap S^*M,$
\begin{eqnarray} \label{q1term}
q_1(x,\xi) &=& 8 (\partial_{\beta_0}\varphi)(x,-2\xi,x)(\xi\cdot\partial_x\beta_0(x,-2\xi)) \nonumber \\
&=& 8 \big(-2 \xi_n + O(\delta) \big) \,\, \xi \cdot \partial_x \beta_0(x) \nonumber\\
&=&  8 \big(-2 \xi_n + O(\delta) \big) \cdot \big(\xi_n + O(\delta) \big) \nonumber \\
&=& - 16 \xi_n^2 + O(\delta), \quad 
\end{eqnarray}

For the $q_2$-term we again use that $|b_0(x,\xi,x)| = 1$ and note that
$$ \partial_{\beta_0} x  = (C_1(x),...,C_{n-1}(x), 1); \quad C_j(x) = O(x);\quad j=1,..,n-1,$$
and
$$ \rho_{\beta_0} = \frac{1}{2} \big( \partial_{x_n} + A(x) \cdot \partial_x \big) ( 1 + O(|x|^2)) |\xi|^2 = O(x) |\xi|^2 = O(\delta),$$
so that

\begin{eqnarray} \label{q2term}
q_2(x,\xi) &=& (\xi\cdot\partial_{\beta_0}x)^2-\rho_{\beta_0}(\xi\cdot\partial_{\beta_0}x) \nonumber \\
&=&  (\xi\cdot\partial_{\beta_0}x)^2 + O(\delta) \nonumber \\
&=& \xi_n^2 + O(\delta).
\end{eqnarray}

It follows from (\ref{q1term}) and (\ref{q2term}) that
$$\int_{S_H^*M} q \, \chi_j d\mu = - 15 \int_{S_H^*M} \xi_n^2 \chi_j d\mu +O(\delta)<0,\quad j=1,..,N$$
for sufficiently small $\delta$, and so, summing over $j=1,...,N,$ it follows that
\begin{equation} \label{qterm}
\int_{S_H^*M} q \, d\mu  \leq -C_0 <0
\end{equation}
for some $C_0>0.$

To complete the proof we note that choosing $\Sigma$ with defining function 
$$ \beta(x,\xi) = \beta_0(x) +\epsilon_0 \, G(x,\xi); \quad G \in C^{\infty}(B^*M),$$
it is clear that
$$ \int_{\Sigma \cap S^*M} q \, d\mu_{\Sigma} = \int_{S_H^*M} q d\mu + O(\epsilon_0).$$
The lemma then follows from (\ref{qterm}) provided $\epsilon_0>0$ is chosen sufficiently small.
\end{proof}

\bigskip

\subsection*{Statements}
\subsubsection*{Conflict of interest} The authors declare that, to the best of their knowledge, there are no conflicts of interest related to the content of this work.
\subsubsection*{Data availability} The authors confirm that no external data were used or are required to support the findings of this study.

\bibliographystyle{alpha}
\bibliography{ref}

\end{document}